\documentclass[10pt]{article}
\usepackage{latexsym,amsfonts,amsthm,amsmath,amssymb,cases}
\usepackage{cite}
\usepackage{enumerate}
\usepackage{graphicx}
\usepackage{color}

\newtheorem{thm}{Theorem}[section]
\newtheorem{cor}{Corollary}[section]
\newtheorem{lm}{Lemma}[section]

\newtheorem{assumption}{Assumption}[section]
\newtheorem{remark}{Remark}[section]

\newtheorem{defn}{Definition}[section]

\newcounter{nextauthor}
\def\mathrm{\mbox}

\numberwithin{remark}{section}

\begin{document}
\title{{\Large \bf Set-valued It\^{o}'s formula with an application to the general set-valued backward stochastic differential equation}\thanks{This work was supported by the National Natural Science Foundation of China (11471230, 11671282).}}
\author{Yao-jia Zhang, Zhun Gou and Nan-jing Huang\thanks{Corresponding author.  E-mail addresses: nanjinghuang@hotmail.com; njhuang@scu.edu.cn}\\
{\small\it Department of Mathematics, Sichuan University, Chengdu,	Sichuan 610064, P.R. China}}
\date{}
\maketitle \vspace*{-9mm}
\begin{center}
\begin{minipage}{5.7in}
\noindent{\bf Abstract.} The overarching goal of this paper is to establish a set-valued It\^{o}'s formula.  As an application,  we obtain the existence and uniqueness of solutions for the general set-valued backward stochastic differential equation which gives an answer to an open question proposed by Ararat et al. (C. Ararat, J. Ma and W.Q. Wu, Set-valued backward stochastic differential equation, arXiv:2007.15073).
\end{minipage}
\end{center}

\section{Introduction}
Set-valued differential equations, both deterministic and stochastic, have attracted the attention of many  scholars due to the wide applications of set-valued functions (mappings) in practical problems \cite{Aubin, Aubin+, Deimling, Kisielewicz2013Stochastic, Tolstonogov}. Recently,  many applications about set-valued deterministic/stochastic differential equations are found in computer science \cite{Xu2018Authenticating}, economics and finance \cite{Li2010Strong,Hamel2011Set-valued, Kisielewicz2013Stochastic, Aswani2019Statistics}.

It\^{o}'s formula, established firstly by It\^{o} \cite{Ito1944Stochastic}, plays an important role in stochastic analysis \cite{Ito1951On,Ito1951Ona}. In order to study some new stochastic differential equations, the It\^{o}'s formula has been extended in several directions.  For example, Applebaum and Hudson \cite{Applebaum1984Fermion} proved an It\^{o} product formula for stochastic integrals against fermion Brownian motion. Al-Hussaini and Elliott \cite{Al-Hussaini1987Nagoya} showed an It\^{o}'s formula for a continuous semi-martingale $X_t$ with a local time $L^\alpha_t$ at $\alpha$. Gradinaru et al. \cite{Gradinaru2005m} gave an It\^{o}'s formula for nonsemimartingales. Catuogno and Olivera \cite{Catuogno2014Time} provided a fresh It\^{o}'s formula for time dependent tempered generalized functions. Recently, a generalised It\^o's formula for L\'{e}vy-driven Volterra Processes was established by Bender et al.  \cite{Bender2015A}. However, to the best of our knowledge, there is no set-valued It\^{o}'s formula. The first goal of this paper is to give a set-valued It\^{o}'s formula.

It is well known that backward stochastic differential equations (BSDEs),  introduced by Pardoux and Peng \cite{Peng1990Adapted}, have been studied extensively in the literature. For instance, we refer the reader to \cite{Hamadene1995Zero-sum,Hu2012Backward,Karoui1997Backward,Li2012Mean-field,Peng2009Backward}. Recently, various examples have been given in the literature \cite{Alexander2018Mean-Field,Borkowski2017Forward,Elliott2012Markovian,Wang2014Linear,Zhang2017Backward} to motivate the study of BSDEs. Very recently, Ararat et al. \cite{Ararat2020Set-valued} provided some sufficient conditions to ensure the existence and uniqueness of solutions for the following set-valued backward stochastic differential equation:
\begin{align*}
Y_t=\xi+\int_t^Tf(t,Y_s)ds \ominus \int_t^TZ_sdW_s,
\end{align*}
where $(W_s,\; s \geqslant 0)$ is a standard $m$-dimension Brownian movement with $dW_idW_j=0(i\neq j)$.

However, as pointed out by Ararat et al. \cite{Ararat2020Set-valued}, the systematic study of the following general set-valued backward stochastic differential equation (GSVBSDE):
\begin{align}\label{e1.1}
Y_t=\xi+\int_t^Tf(t,Y_s,Z_s)ds \ominus \int_t^TZ_sdW_s,
\end{align}
is still widely open. It is worth mentioning that \eqref{e1.1} is a modelling tool used to capture the risk measure problem arising in finance \cite{Ararat2020Set-valued}.  The second purpose of this paper is to study the existence and uniqueness of \eqref{e1.1} by using the set-valued It\^{o}'s formula.

The rest of this paper is structured as follows. The second section recalls some necessary preliminaries, including some properties of Hukuhara difference, set-valued stochastic processes and stochastic integrals. After that in Section 3, we obtain the set-valued it\^{o}'s formula by employing some properties of set-valued stochastic integrals. Finally, we show the existence and uniqueness of the solutions to \eqref{e1.1} as the application of the set-valued It\^{o}'s formula.

\section{Preliminaries}
\noindent \setcounter{equation}{0}
In this section, we recall some necessary notations and definitions.
\subsection{Hukuhara difference}
For a Hilbert space $\mathbb{X}$, let $\mathcal{P}(\mathbb{X})$ be the set of all nonempty subsets of $\mathbb{X}$.  Let $\mathcal{L}(\mathbb{X})$ be the set of all closed sets in $\mathcal{P}(\mathbb{X})$ and $\mathcal{K}(\mathbb{X})$ be the set of all compact convex sets in $\mathcal{P}(\mathbb{X})$.

For any $A,B \in \mathcal{K}(\mathbb{X})$, the Hausdorff distance between $A$ and $B$ is defined by
$$h(A,B):=\max\left\{\sup_{x\in A}\inf_{y\in B}d(x,y),\sup_{x\in B}\inf_{y\in A}d(x,y)\right\}$$
and the mapping $\|\cdot\|:  \mathcal{K}(\mathbb{X})\to [0,\infty)$ is defined by
\begin{eqnarray}\label{e1+}
\|A\|:=h(A,\left\{0\right\})=\sup\limits_{a \in A }|a|,\quad \forall A\in \mathcal{K}(\mathbb{X}).
\end{eqnarray}
Moreover, for any $A,B \in \mathcal{K}(\mathbb{R}^n)$ and $\alpha \in \mathbb{R}$, define
$$
A+B:=\left\{a+b:a\in A,b\in B\right\};\qquad \alpha A:=\left\{\alpha a : a \in A\right\}.
$$

Clearly, for $A,B,C\in\mathcal{K}(\mathbb{R}^n)$, the following cancellation law holds:
$$
A+C=B+C \Longleftrightarrow  A=B.
$$

Assume that $(\Omega,\mathcal{F},\mathbb{P})$ is a probability measure space. $A \in \mathcal{F}$ is called an atomic set, if $\mathbb{P}(A)>0$ implies that $\mathbb{P}(B)=0$ or $\mathbb{P}(A\setminus B)=0$ for any Borel subset $B\subset A$. If there is no atomic sets under the measure $\mathbb{P}$, then $\mathbb{P}$ is called a nonatomic probability measure. In this paper, we always assume $\mathbb{P}$ is nonatomic. In the sequel, we assume that $(\Omega,\mathcal{F},\mathbb{F},\mathbb{P})$ is a filtered probability space satisfying the usual conditions.

Now we recall the Hukuhara difference, which defines the ``subtraction" in the space $\mathcal{K}(\mathbb{X})$.
\begin{defn}\label{d2.1} \cite{Hukuhara1967Integration}
The Hukuhara difference is defined as follows:
$$
A \ominus B =C \Longleftrightarrow A=B+C,\quad A,B\in \mathcal{K}(\mathbb{X}).
$$
\end{defn}

For our main results, we need the following lemmas.

\begin{lm}\label{12.1}  \cite{Ararat2020Set-valued}
Let $A,B,A_1,A_2,B_1,B_2,C\in \mathcal{K}(\mathbb{R}^n)$. The following identities hold if all the Hukuhara differences involved exist.
\begin{enumerate}[($\romannumeral1$)]
\item $A\ominus A=\left\{ 0 \right\}$, $A\ominus \left\{ 0 \right\}$=$A$;
\item $(A_1+B_1)\ominus (A_2+B_2)=(A_1\ominus A_2)+(B_1\ominus B_2)$;
\item $(A_1+B_1)\ominus B_2 =A_1+(B_1\ominus B_2)=(A_1\ominus B_2)+B_1$;
\item $A_1+(B_1\ominus B_2)=(A_1 \ominus B_2)+B_1$;
\item $A=B+(A\ominus B)$.
\end{enumerate}
\end{lm}

\begin{lm}\label{l2.2} Let $A,B,C\in \mathcal{K}(\mathbb{R}^n)$. Then
\begin{enumerate}[($\romannumeral1$)]
\item  If $A\ominus C$ exists, then $(A+B)\ominus C$ exists for any $B\in \mathcal{K}(\mathbb{R}^n)$;
\item If $A\ominus C$ and $C\ominus B$ exist, then $A\ominus B$ exists and $A\ominus B=(A\ominus C)+(C\ominus B)$.
\end{enumerate}
\end{lm}

\begin{proof}
(i) Set $M=A\ominus C $.  Then $A=C+M$ and so $(A+B)\ominus C= (C+M+B)\ominus C=M+B$. (ii)
Let $M=A\ominus C$, $N=C\ominus B$, then $A=C+M$, $C=B+N$. Taking $C=B+N$ in $A=C+M$, we have $A=B+N+M$, which implies $A\ominus B=M+N=(A\ominus C)+(C\ominus B)$.
\end{proof}

\begin{lm}\cite{Ararat2020Set-valued}\label{l2.4}
\begin{enumerate}[($\romannumeral1$)]
\item The mapping $\|\cdot\|:  \mathcal{K}(\mathbb{X})\to [0,\infty)$ defined by \eqref{e1+} satisfies the properties of a norm;
\item If $A,B \in \mathcal{K}(\mathbb{X})$ and $A\ominus B$ exists, then $h(A,B)=\|A\ominus B\|$;
\item For any $A,B \in \mathcal{K}(\mathbb{R}^n)$,  both $A\ominus B$ and $B\ominus A$ exist if and only if $A$ is a translation of $B$, i.e., $A=B+\left\{c\right\}$,  where $c\in \mathbb{R}^n$.
\end{enumerate}
\end{lm}

\subsection{Set-valued stochastic processes}
In this subsection, we recall the set-valued stochastic processes and give a limit convergence theorem for set-valued processes.

For a sub-$\sigma $-filed $\mathcal{G} \subset \mathcal{F}$, let $\mathbb{L}^0_{\mathcal{G}}(\Omega,\mathcal{K}(\mathbb{R}^n))$ denote the set of all $\mathcal{G}$-measurable random variables valued in $\mathcal{K}(\mathbb{R}^n)$. For any $X\in \mathcal{K}(\mathbb{R}^n)$, let $S_{\mathcal{G}}(X)$ be the set of all $\mathcal{G}$-measurable selection of $X$. Moreover, let $\mathbb{L}^2_{\mathcal{G}}(\Omega,\mathcal{K}(\mathbb{R}^n))$ denote the set of all $\mathcal{G}$-measurable square-integrably bounded random variables valued in $\mathcal{K}(\mathbb{R}^n)$. For simplicity, let $S^2_{\mathcal{G}}(X)=S_{\mathcal{G}}(X) \bigcap  \mathbb{L}^2_{\mathcal{G}}(\Omega,\mathcal{K}(\mathbb{R}^n))$.

\begin{defn} \label{a} \cite{Kisielewicz2013Stochastic}
\begin{itemize}
\item $L^2_{ad}(\Omega ,\mathbb{R}^n)=L^2_{ad}(\Omega,\mathcal{F},\mathbb{P},\mathbb{R}^n)$, the set of all the $n$-dimensional measurable square-integrably random variables,  is a Hilbert space equipped with the norm $\|\cdot\|_a=(E|\cdot|^2)^{\frac{1}{2}}$.

\item $L^2_{ad}([0,T] \times \Omega ,\mathbb{R}^n)=\left\{ f(t,\omega) | f(t,\omega) \;\mbox{is} \; \mathcal{F}_t\mbox{-adapted}\; \mbox{with}  \;\mathbb{E}\left[\int_0^T|f(t,\omega)|^2dt\right]< \infty\right\}$ is a Hilbert space with the norm $\| f(t,\omega) \|_c=\mathbb{E} \left[\int_0^T |f(t,\omega)|^2dt\right]^{\frac{1}{2}}$.
\item A set-valued mapping $F: \Omega \rightarrow  \mathcal{L}(\mathbb{X})$ is called measurable if,  for any closed set $B \subset \mathbb{X}$, it holds that $\left\{ \omega \in \Omega : F(\omega) \cap B \neq \emptyset\right\} \in \mathcal{F}$.

\item A measurable set-valued mapping $X: \Omega \rightarrow \mathcal{L}(\mathbb{R}^n)$ is called a set-valued random variable.

\item A set-valued stochastic process $f=\left\{ f_t \right\}_{t\in [0,T]}$ is a family of set-valued random variables taking values in $\mathcal{L}(\mathbb{R}^n)$.

\item A set-valued stochastic process is called measurable if it is $\mathcal{B}([0,T])\otimes \mathcal{F}$-measurable as a single function on $[0,T]\times \Omega$.

\item $\mathbb{L}^0(\Omega,\mathcal{L}(\mathbb{R}^n))==\mathbb{L}^0(\Omega,\mathcal{F}_T,\mathbb{P},\mathcal{L}(\mathbb{R}^n))$, is the space of all measurable set-valued mappings $F:E\rightarrow \mathcal{L}(\mathbb{R}^n)$ distinguished up to $\mathbb{P}$-almost everywhere equality.

\item $\mathbb{L}^0_t(\Omega,\mathcal{L}(\mathbb{R}^n))$ is the set of all $\mathcal{F}_t$-measurable random variables valued in $\mathcal{L}(\mathbb{R}^n)$ and $\mathbb{L}^p_t(\Omega,\mathcal{L}(\mathbb{R}^n))$ is the set of all $\mathcal{F}_t$-measurable $p$-integrably bounded random variables valued in $\mathcal{L}(\mathbb{R}^n)$ for any given $t\in [0,T]$.

\item $\mathcal{A}^2(\Omega,\mathcal{L}(\mathbb{R}^d)):=\left\{ F\in \mathbb{L}^0(\Omega,\mathcal{L}(\mathbb{R}^n)): S^{*2}(F)\neq \emptyset\right\}$ and
    $$\mathcal{A}_t^2(\Omega,\mathcal{L}(\mathbb{R}^d)):=\left\{ F\in \mathbb{L}_t^0(\Omega,\mathcal{L}(\mathbb{R}^n)): S^{*2}(F)\neq \emptyset\right\}.$$

\item For any $X\in \mathcal{A}^2(\Omega,\mathcal{L}(\mathbb{R}^d))$ and any $t \in [0,T]$, $S^*_t(X)$ is the set of all $\mathcal{F}_t$-measurable selections of $X$.

\item For any $X \in \mathbb{L}^0(\Omega,\mathcal{L}(\mathbb{R}^n))$ the collection of all the selections of $L^2_{ad}(\Omega ,\mathcal{R}^n)$ is denoted by
    $$S^{*2}(X)=\left\{f\in L^2_{ad}(\Omega ,\mathbb{R}^n): f(\omega) \in F(\omega) \quad \mathbb{P}\mbox{-}a.s.\right\}.$$

\item $S^{*2}_T(X)=S^*_T(X)\bigcap L^2_{ad}(\Omega ,\mathbb{R}^n)$.

\item A set-valued random variable $X(\omega)$ is called $p$-integrably bounded if there exists $m \in L^2_{ad}(\Omega, \mathbb{R}^+)$ such that $\| X\| \leqslant m(\omega)$ $\mathbb{P}$-a.s.

\item $L^p(E,\mathcal{L}(\mathbb{R}^n))=\left\{F \in L^0(E,\mathcal{L}(\mathbb{R}^n)): \|F\| \leqslant m(\omega)  \quad \mathbb{P}\mbox{-}a.s.\right\}$, where $m(\omega) \in L^p_{ad}(\Omega , \mathbb{R}^+)$.

\item A set-valued stochastic process $F$ is called $p$-integrably bounded if there exists $m \in L^2_{ad}([0,T] \times \Omega, \mathbb{R}^+)$ such that $\| F\| \leqslant m(t,\omega)$ $\mathbb{P}$-a.s..

\item Denote the subtrajectory integrals of a set-valued stochastic process $F:T\times \Omega \rightarrow \mathcal{L}(\mathbb{R}^n)$ by $S(F)$, which is the set of all measurable and $dt \times P$-integrable selectors of $F$.

\item Denote the subset
$$S^2(F):=S(F) \bigcap L^2_{ad}([0,T] \times \Omega ,\mathcal{R}^n)$$
and
$$S^2_{\mathbb{F}}(F)=\left\{ f \in S^2(F): f \; is \; \mathbb{F}\mbox{-}measurable \right\}.$$

\end{itemize}
\end{defn}

\begin{defn} \cite{Kisielewicz2013Stochastic}
A set-valued process is called $\mathbb{F}$-nonanticipative if it is $\Sigma_{\mathbb{F}}$-measurable, where $$\Sigma_{\mathbb{F}}=\left\{A\in \beta_T \otimes \mathcal{F}:A_t \in \mathcal{F}_t \quad \forall t\in T\right\},$$
and $A_t$ denotes the $t$-section of a set $A \subset T\times \Omega$, i.e., $A_t=\{\omega\in \beta_T: A(t,\omega)\in \beta_T \otimes \mathcal{F}\}$. The set of all set-valued $\mathbb{F}$-nonanticipative measurable square integrably bounded stochastic processes taking values in $\mathcal{K}(\mathbb{R})$ with $S^2_{\mathbb{F}}(\cdot) \neq \emptyset$  is denoted by $\mathcal{L}^2_{ad}([0,T]\times \Omega,\mathcal{K}(\mathbb{R}^n))$.

\end{defn}

\begin{lm}\cite{Kisielewicz2013Stochastic}\label{l2.5}
Let $\Phi_t,\,\Psi_t \in \mathcal{L}^2_{ad}([0,T]\times \Omega,\mathcal{L}(\mathbb{R}^n))$.  Then $\mathcal{L}^2_{ad}([0,T]\times \Omega,\mathcal{L}(\mathbb{R}^n))$ is a complete metric space with the metric $d_H(\Phi_t,\Psi_t)=[\mathbb{E}\int_0^Th^2(\Phi_t,\Psi_t)dt]^{\frac{1}{2}}$. Specially,  $\mathcal{L}^2_{ad}([0,T]\times \Omega,\mathcal{K}(\mathbb{R}^n))$ is a Banach space with the norm $\|Z(t)(\omega)\|_s=[\mathbb{E}\int_0^T\|Z\|dt]^{\frac{1}{2}}$ for $Z(t)\in \mathcal{L}^2_{ad}([0,T]\times \Omega,\mathcal{K}(\mathbb{R}^n))$.
\end{lm}

\begin{lm}(Kuratowski and Ryll-Nardzewski)\cite{KR65}
Assume that $(E,\mathcal{E})$ is a $Polish$ space and $(T,\mathcal{F})$ is a measurable space. If $F: T\rightarrow \mathcal{L}(E)$ is measurable, then $F$ admits a measurable selector.
\end{lm}

The set-valued conditional expectation and the set-valued martingale can be defined as follows.

\begin{defn}\cite{Kisielewicz2013Stochastic}
For a sub-$\sigma $-field $\mathcal{G} \subset \mathcal{F}$ and $X\in \mathbb{L}^2(\Omega,\mathcal{K}(\mathbb{R}^n))$, the conditional expectation of $X$ with respect $\mathcal{G}$ is defined as the unique set-valued random variable $\mathbb{E}[X|\mathcal{G}] \in \mathbb{L}^2_{\mathcal{G}}(\Omega,\mathcal{K}(\mathbb{R}^n))$ such that
$$\int_{G}\mathbb{E}[X|\mathcal{G}]d \mathbb{P}=\int_{G}Xd \mathbb{P},\qquad \forall G\in \mathcal{G}.$$
\end{defn}

\begin{defn}\cite{Kisielewicz2014Martingale}\label{d25}
A set-valued process $M=\left\{M_t\right\}_{t\in [0,T]}$ is said to be a set-valued square-integrable $\mathcal{F}$-martingale if
\begin{enumerate}[($\romannumeral1$)]
\item $M\in \mathcal{L}^2_{ad}([0,T]\times \Omega,\mathcal{L}(\mathbb{R}^n))$;
\item $M_t\in \mathcal{A}^2_{t}(\Omega,\mathcal{L}(\mathbb{R}^n))$;
\item $M_s=\mathbb{E}[M_t|\mathcal{F}_s]$ for all $0\leqslant s \leqslant t$.
\end{enumerate}
\end{defn}

\begin{lm}\cite{Ararat2020Set-valued}
Let $F_1,F_2\in \mathcal{L}^2_{ad}([0,T]\times \Omega,\mathcal{K}(\mathbb{R}^n))$. Then, $F_1+F_2\in \mathcal{L}^2_{ad}([0,T]\times \Omega,\mathcal{K}(\mathbb{R}^n))$ and
$$
S_{\mathbb{F}}^2(F_1+F_2)=S_{\mathbb{F}}^2(F_1)+S_{\mathbb{F}}^2(F_2).
$$
Furthermore, if $F_1 \ominus F_2$ exists, then $F_1 \ominus F_2 \in \mathcal{L}^2_{ad}([0,T]\times \Omega,\mathcal{K}(\mathbb{R}^n))$ and
$$
S_{\mathbb{F}}^2(F_1\ominus F_2)=S_{\mathbb{F}}^2(F_1) \ominus S^2_{\mathbb{F}}(F_2).
$$
\end{lm}

\begin{lm}\label{t32}
Suppose that $\left\{X_p(t,\omega)\right\}_{p=1}^{\infty} \subset \mathcal{L}^2_{ad}([0,T]\times \Omega,\mathcal{K}(\mathbb{R}^n))$ is a sequence of set-valued stochastic processes satisfying
\begin{enumerate}[($\romannumeral1$)]
\item $X_q \ominus X_p$ exists for  $q \geqslant p$;
\item For each $\varepsilon >0$, there exists an $N\in \mathbb{N}$ such that $\|X_q\ominus X_p\|_s <\varepsilon$ when $p,q \geqslant N$.
\end{enumerate}
Then there exists a unique $X(t,\omega) \in \mathcal{L}^2_{ad}([0,T]\times \Omega,\mathcal{K}(\mathbb{R}^n))$ such that $\|X_p\ominus X\|_s \rightarrow 0$ as $p\rightarrow +\infty$.
\end{lm}

\begin{proof}
It follows from the assumption (i) that the potential of the $X_p$ is increasing as $p$ becoming larger. Thus, the assumption (ii) shows that, for each $\varepsilon >0$, there exist an $N'\in \mathbb{N}$ and a $\xi\in \mathbb{R}^p$ such that $X_q\ominus X_p={\xi}$ with $|\xi|<\varepsilon$ when $n,m \geqslant N'$. Since $X_q \ominus X_p$ exists, we have
$$\|X_q \ominus X_p\|_s=\left[E\int_0^T\|X_q \ominus X_p\|^2ds\right]^{\frac{1}{2}}=\left[E\int_0^Th^2(X_q, X_p)ds\right]^{\frac{1}{2}}$$
when $p,q \geqslant N'$. By Lemma \ref{l2.5} and the fact $\left\{X_p(t,\omega)\right\}_{p=1}^{\infty}$ in $\mathcal{L}^2_{ad}([0,T]\times \Omega,\mathcal{K}(\mathbb{R}^n))$, there exists $X^*\in \mathcal{L}^2_{ad}([0,T]\times \Omega,\mathcal{K}(\mathbb{R}^n))$ such that $\|X^*\ominus X_p\|_s <\varepsilon$ for all $t\in [0,T]$ and a.s $\omega\in \Omega$ when $p\geqslant N'$.

Next we show the uniqueness. If $X_1,X_2 \in \mathcal{L}^2_{ad}([0,T]\times \Omega,\mathcal{K}(\mathbb{R}^n))$ such that $\|X_p\ominus X_1\|_s \rightarrow 0$ and $\|X_p\ominus X_2\|_s \rightarrow 0$ with $p\rightarrow +\infty$, then it follows from Lemma \ref{l2.2} that
\begin{align*}
 \|X_1\ominus X_2\|_s&=\|X_1\ominus X_p + X_p\ominus X_2\|_s \\
 &\leqslant \|X_p\ominus X_1\|_s +\|X_p\ominus X_2\|_s \rightarrow 0
\end{align*}
as $p\rightarrow +\infty$. This ends the proof.
\end{proof}

\subsection{Set-valued stochastic integrals}
In this subsection, we recall the the notions of set-valued stochastic integrals.

\begin{defn} (\cite{Ararat2020Set-valued,Kisielewicz2013Stochastic})
\begin{itemize}
\item For $F\in \mathcal{L}^2_{ad}([0,T]\times \Omega,\mathcal{K}(\mathbb{R}^n))$, \\
$\mathcal{M}_{\mathcal{F}}([0,T]\times \Omega,\mathcal{K}(\mathbb{R}^n)):=\left\{ F \in \mathcal{L}^2_{ad}([0,T]\times \Omega,\mathcal{K}(\mathbb{R}^n)):S^2_{\mathbb{F}}(F) \neq \emptyset \right\}$.

\item A set $V\subset L^2_{ad}(\Omega ,\mathcal{R}^n)$ is said to be decomposable with respect to $\mathcal{F}$ if for any $f_1,f_2 \in V$ and $D\in \mathcal{F}$, it holds that $1_Df_1+1_{D^c}f_2 \in V$.

\item Denote the decomposable hull of $V\subset L^2_{ad}(\Omega ,\mathcal{R}^n)$ by $dec(V)$, which is the smallest decomposable subset of $L^2_{ad}(\Omega ,\mathcal{R}^n)$ contain $V$. Denote the closure of $dec(V)$ in $L^2_{ad}(\Omega ,\mathcal{R}^n)$  by $\overline{dec}(V)$.

\item $\mathcal{K}_w(\mathcal{L}^2_{ad}([0,T] \times \Omega ,\mathcal{R}^{n\times m})$ is the set of weakly compact convex subset of $\mathcal{L}^2_{ad}([0,T] \times \Omega ,\mathcal{R}^{n\times m})$.
\end{itemize}
\end{defn}

\begin{lm}\cite{Revuz2008Continuous} \label{itoi}
Suppose $f\in L^2_{ad}([0,T] \times \Omega ,\mathbb{R}^m)$. Then the It\^{o} integral $I(f)=\int_0^T f(t)dW(t)$ is a random variable with $\mathbb{E}[I(f)]=0$ and
$$\mathbb{E}\left[I(t)\right]^2= \mathbb{E}\left[\int_0^Tf^2(t)dt\right].$$
\end{lm}

\begin{defn} \cite{Ararat2020Set-valued,Kisielewicz2013Stochastic}
\begin{itemize}

\item Denote the $J:\mathcal{L}^2_{ad}([0,T]\times \Omega ,\beta _T \otimes \mathcal{F}_T,\mathbb{R}^n)\rightarrow \mathcal{L}^2_{ad}(\Omega,\mathcal{F}_T,\mathbb{R}^n)$ by
$$J(\phi)(\omega)=\left(\int_0^T \phi(t) dt\right)(\omega) \quad \mathbb{P}\mbox{-}a.s..$$
where  $\phi \in \mathcal{L}^2_{ad}([0,T]\times \Omega,\beta _T \otimes \mathcal{F}_T,\mathbb{R}^n)$, $\beta _T$ and $\mathcal{F}_T$ are the $\sigma$-fields on $T$ and $\Omega$, respectively.

\item Denote the
$\mathcal{J}:\mathcal{L}^2_{ad}([0,T]\times \Omega ,\Sigma_{\mathbb{F}},\mathbb{R}^{n\times m})\rightarrow \mathcal{L}^2_{ad}(\Omega,\mathcal{F}_T,\mathbb{R}^n)$ by
$$\mathcal{J}(\psi)(\omega)=\left(\int_0^T \psi(t) dW_t\right)(\omega) \quad \mathbb{P}\mbox{-}a.s..$$
where $\psi \in \mathcal{L}^2([0,T]\times \Omega,\Sigma_{\mathbb{F}},\mathbb{R}^{n\times m})$.
\item Denote the sets $J[1_{[s,t]}S(F)]$ and $\mathcal{J}[1_{[s,t]}S(F)]$ by $J_{st}[S_{\mathbb{F}}(F)]$ and $\mathcal{J}_{st}[S_{\mathbb{F}}(G)]$, respectively, which are said to be the functional set-valued stochastic integral of $F$ and $G$ on the interval $[s,t]$, respectively.
\end{itemize}
\end{defn}

\begin{defn} \cite{Kisielewicz2013Stochastic}
Let $F\in \mathcal{L}^2_{ad}([0,T]\times \Omega,\mathcal{L}(\mathbb{R}^n))$ and $G\in \mathcal{L}^2_{ad}([0,T]\times \Omega,\mathcal{L}(\mathbb{R}^{n\times m}))$ such that $S_{\mathbb{F}}^2(F)\neq \emptyset $ and $S_{\mathbb{F}}^2(G)\neq \emptyset$. Then there exist unique set-valued random variables $\int^T_0 F(t) dt$ and $\int^T_0 G(t)dW_t$ in $\mathcal{A}_T^2(E,\mathcal{L}(\mathbb{R}^d))$ such that $S^2_T(\int^T_0 F(t) dt)=\overline{dec}(J[S^2_{\mathbb{F}}(F)])$ and $S^2_T(\int^T_0 G(t)dW_t)=\overline{dec}(\mathcal{J}[S^2_{\mathbb{F}}(G)])$. The set-valued random variables $\int^T_0 F(t) dt$ and $\int^T_0 G(t)dW_t$ are, respectively, called the set-valued stochastic integrals of $F(t)$ and $G(t)$ with respect to $t$ and $W_t$ on $[0,T]$. Moreover, for any $t\in [0,T]$, define
$$\int^t_0 F(s) ds:=\int^T_0 1_{(0,t]}(s)F(s) ds, \quad \int^t_0 G(s)dW_s:=\int^T_0 1_{(0,t]}(s)G(s)dW_s$$
and
$$\int_t^T F(s) ds:=\int^T_0 1_{(t,T]}(s)F(s) ds, \quad \int^T_t G(s)dW_s:=\int^T_0 1_{(t,T]}(s)G(s)dW_s.$$
\end{defn}

In the sequel, we always suppose that the set-valued integrals exist.
\begin{lm}\cite{Ararat2020Set-valued}
Let $F_1,F_2 \in \mathcal{L}^2_{ad}([0,T]\times \Omega,\mathcal{K}(\mathbb{R}^n))$ and $G_1,G_2 \in\mathcal{L}^2_{ad}([0,T]\times \Omega,\mathcal{K}(\mathbb{R}^{n\times m}))$. Then, for every $t\in [0,T]$,
$$\int_0^t(F_1(s)+F_2(s))ds=\int_0^t F_1(s)ds+\int_0^t F_2(s)ds$$
and
$$\int_0^t(G_1(s)+G_2(s))dW_s=\int_0^tG_1(s)dW_s+\int_0^tG_2(s)dW_s$$
hold $\mathbb{P}$-a.s.. If $F_1\ominus F_2$ and $G_1\ominus G_2$ exist, then $F_1\ominus F_2\in\mathcal{L}^2_{ad}([0,T]\times \Omega,\mathcal{K}(\mathbb{R}^n))$ and $G_1\ominus G_2\in \mathcal{L}^2_{ad}([0,T]\times \Omega,\mathcal{K}(\mathbb{R}^{n\times m}))$ for every $t\in [0,T]$. Moreover,
$$\int_0^t F_1(s)\ominus F_2(s) ds= \int_0^t F_1(s)ds \ominus \int_0^t F_2(s)ds$$
and
$$\int_0^tG_1(s)\ominus G_2(s)dW_s=\int_0^tG_1(s)dW_s\ominus \int_0^tG_2(s)dW_s$$
hold $\mathbb{P}$-a.s..
\end{lm}

\begin{lm}\cite{Ararat2020Set-valued}\label{42}
Let $F\in \mathcal{L}^2_{ad}([0,T]\times \Omega,\mathcal{K}(\mathbb{R}^n))$ and $G\in \mathcal{L}^2_{ad}([0,T]\times \Omega,\mathcal{K}(\mathbb{R}^{n\times m}))$. Then for each $t\in [0,T]$,
$$\int_0^T F(s)ds=\int^t_0 F(s)ds+\int_t^T F(s)ds,\quad \int_0^T G(s)dW_s=\int_0^t G(s)dW_s+\int_t^T G(s)dW_s$$
and
$$\int_t^TF(s)ds=\int^T_0 F(s)ds\ominus \int^t_0 F(s)ds,\quad \int_t^TG(s)dW_s=\int_0^TG(s)dW_s \ominus \int_0^tG(s)dW_s$$
hold $\mathbb{P}$-a.s..
\end{lm}

\begin{lm}\cite{Ararat2020Set-valued} \label{l2.11}
For any $Z,Z_1,Z_2\in \mathcal{K}_w(\mathcal{L}^2_{ad}([0,T] \times \Omega ,\mathcal{R}^{n\times m}))$, the following statements hold:
\begin{enumerate}[($\romannumeral1$)]
\item $Z_1+Z_2 \in \mathcal{K}_w(\mathcal{L}^2_{ad}([0,T] \times \Omega ,\mathcal{R}^{n\times m}))$ and for every $t\in [0,T]$,
$$\int_0^t(Z_1+Z_2)dW_s=\int_0^t Z_1dW_s+\int_0^t Z_2dW_s,\quad \mathbb{P}\mbox{-}a.s.;$$
\item If $Z_1\ominus Z_2$ exists, then $Z_1\ominus Z_2 \in \mathcal{K}_w(\mathcal{L}^2_{ad}([0,T] \times \Omega ,\mathcal{R}^{n\times m}))$ and for every $t\in [0,T]$,
$$\int_0^t(Z_1 \ominus Z_2)dW_s=\int_0^t Z_1dW_s\ominus \int_0^t Z_2dW_s,\quad \mathbb{P}\mbox{-}a.s.;$$
\item If $Z_1\ominus Z_2$ exists and $\int_0^t Z_1dW_s=\int_0^t Z_2dW_s,\; \mathbb{P}$-a.s. for all $t\in [0,T]$, then $Z_1=Z_2$ $\mathbb{P}$-a.s.;
\item $$\int_0^T ZdW_s=\int_0^t ZdW_s+\int_t^T Z dW_s,\quad \int_t^T ZdW_s=\int_0^T ZdW_s\ominus \int_0^t Z dW_s.$$
\end{enumerate}
\end{lm}

\begin{lm} \cite{Kisielewicz2014Martingale} \label{6}
For every convex-valued square-integrable set-valued martingale $M=\left\{M_t\right\}_{t\in [0,T]}$, there exists $\mathcal{G}^M \in \mathcal{P}(\mathcal{L}^2_{ad}([0,T]\times \Omega,\mathbb{R}^{n\times m})$ such that $M_t=M_0+\int_0^t \mathcal{G}^M dB_s$ $\mathbb{P}$-a.s. for all $t\in[0,T]$. Moreover, if $M$ is uniformly square-integrably bounded, then $\mathcal{G}^M$ is a convex weakly compact set, i.e., $\mathcal{G}^M \in \mathcal{K}_w(\mathcal{L}^2_{ad}([0,T] \times \Omega,\mathcal{R}^{n\times m}))$.
\end{lm}

Denote $\mathbb{G}$ by
$$\mathbb{G}:=\left\{ Z(t) \in \mathcal{K}_w(\mathcal{L}^2_{ad}([0,T] \times \Omega ,\mathcal{R}^{n\times m})):\left\{\int_0^t Z(s) dW_s\right\}_{t\in [0,T]} \;\mbox{is an} \; \mathcal{F}_t \;\mbox{martingale}\right\}.$$

\begin{remark}\label{r21}
Since $\mathcal{K}_w(\mathcal{L}^2_{ad}([0,T] \times \Omega ,\mathcal{R}^{n\times m}))$ is weakly closed and bounded, $\mathbb{G}$ is closed.
\end{remark}

\begin{lm} \cite{Ararat2020Set-valued} \label{5}
Let $X_1, X_2 \in \mathbb{L}_T^2(\Omega,\mathcal{K}(\mathbb{R}^n))$ and $\mathcal{G} \subset \mathcal{F}$ be a sub-$\sigma$-field. If $X_1\ominus X_2$ exists, then $E[X_1\ominus X_2|\mathcal{G}]$ exists in $\mathbb{L}_{G}^2(\Omega,\mathcal{K}(\mathbb{R}^n))$ and
$$E[X_1\ominus X_2|\mathcal{G}]=E[X_1|\mathcal{G}]\ominus E[X_2|\mathcal{G}].$$
\end{lm}

\section{Set-valued It\^{o}'s formula}
\noindent \setcounter{equation}{0}
In this section, we extend the classical It\^{o}'s formula to the set-valued case. To this end, we need the following lemma.

\begin{lm}\cite{Kisielewicz2013Stochastic}\label{7}
\begin{enumerate}[($\romannumeral1$)]
\item If $(\Omega,\mathcal{F},\mathbb{P})$ is separable and $\Phi_t \in \mathcal{L}^2_{ad}([0,T]\times \Omega,\mathcal{L}(\mathbb{R}^{n\times m}))$, then there exists a sequence $\left\{\phi^n\right\}_{n=1}^{\infty} \subset S_{\mathbb{F}}^2(\Phi)$ such that
    $$\left(\int_0^T \Phi_s dW_s\right)(\omega)=cl_{\mathbb{L}}\left\{\left(\int_0^T \phi_t^ndW_s\right)(\omega):n\geqslant 1\right\} \quad \mathbb{P}\mbox{-}a.s.,$$
 where $cl_{\mathbb{L}}$ means the closure taken in the norm topology of $\mathcal{L}^2_{ad}([0,T]\times \Omega,\mathcal{L}(\mathbb{R}^n))$.
\item If $(\Omega,\mathcal{F},\mathbb{P})$ is separable and $\Psi_t \in \mathcal{L}^2_{ad}([0,T]\times \Omega,\mathcal{L}(\mathbb{R}^n))$, then there exists a sequence $\left\{\psi^n\right\}_{n=1}^{\infty} \subset S_{\mathbb{F}}^2(\Psi)$ such that
    $$\left(\int_0^T \Psi_s ds\right)(\omega)=cl_{\mathbb{L}}\left\{\left(\int_0^T \psi_t^nds\right)(\omega):n\geqslant 1\right\}\quad \mathbb{P}\mbox{-}a.s..$$
    \end{enumerate}
\end{lm}

We define a set-valued It\^{o}'s process as follows:
\begin{align}\label{e3.1}
X_t=x_{0}+\int_{0}^t f(s)dW_s+\int_0^tg(s)ds,\qquad 0\leqslant t\leqslant T,
\end{align}
where $x_0=X_0\in L^2_{ad}(\Omega ,\mathbb{R}^n)$, $f(t) \in \mathcal{L}^2_{ad}([0,T]\times \Omega,\mathcal{K}(\mathbb{R}^{n\times m}))$ and $g(t) \in \mathcal{L}^2_{ad}([0,T]\times \Omega,\mathcal{K}(\mathbb{R}^n))$.

Clearly, Lemma \ref{7} shows that the set-valued It\^{o}'s process defined by \eqref{e3.1} is a set-valued stochastic process and \eqref{e3.1} can be rewritten as follows:
\begin{align}\label{e3.2}
X_i(t)=x_{0i}+\int_0^tf_i(s)dW_s+\int_0^tg_i(s)ds,\qquad i=1,\cdots,n,
\end{align}
where $f^T_i(t) \in \mathcal{L}^2_{ad}([0,T]\times \Omega,\mathcal{K}(\mathbb{R}^m))$ and $g_i(t) \in \mathcal{L}^2_{ad}([0,T]\times \Omega,\mathcal{K}(\mathbb{R}))$ are the $i$-th component of $f(t)$ and $g(t)$, respectively.

\begin{lm}\label{lm3.3}
Let $A=\left\{(\int_0^t \phi_t^ndW_s)(\omega):n\geqslant 1\right\}$ and $B=\left\{(\int_0^t \psi_t^mds)(\omega):m\geqslant 1\right\}$ \; $\mathbb{P}$-a.s., where $\left\{\phi_t^n\right\}_{n=1}^{\infty} \subset S_{\mathbb{F}}^2(\Phi)$ and $\left\{\psi_t^n\right\}_{n=1}^{\infty} \subset S_{\mathbb{F}}^2(\Psi)$. Then $cl_{\mathbb{L}}(A)+cl_{\mathbb{L}}(B)= cl_{\mathbb{L}}(A+B)$.
\end{lm}
\begin{proof}
Let $\overline{A}=cl_{\mathbb{L}}(A)$ and $\overline{B}=cl_{\mathbb{L}}(B)$. Then we need to show that $\overline{A}+\overline{B}= \overline{A+B}$.

Step 1.
For any given $x\in \overline{A}+\overline{B}$, there exists $a\in \overline{A}$ and $b\in \overline{B}$ such that $x=a+b$.
Moreover, there exist $\left\{a_n\right\} \subset A$ and $\left\{b_n\right\} \subset B$ such that $a_n\to a$ and $b_n \to b$. Thus, $a_n+b_n\to a+b$. Since $a_n+b_n \in A+B$, we have $x=a+b\in \overline{A+B}$, which implies $\overline{A}+\overline{B} \subset \overline{A+B}$.

Step 2.
For any given  $x\in \overline{A+B}$, there exists $\{x_n\}\subset A+B$ such that $x_n \to x$. This means that, for any given $\varepsilon >0$, there exists a positive integer $N_0>0$ such that $\|x_n-x_m\|_c<\varepsilon$ when $n,m>N_0$.  Since $x_n\in A+B$, there exist $a_n=\int_0^t f^n(s)dW_s \in A$ and $b_n=\int_0^t g^n(s)ds \in B$ such that $x_n=a_n+b_n$ such that
$$\int_0^t f^n(s)dW_s + \int_0^t g^n(s)ds \to x.$$
If there is no limit for $\{a_n\}$, then that exists a $\varepsilon_0>0$ such that, for any $N>1$, there exists $n,m >N$ satisfying $\|a_n-a_m\|_c\geqslant \varepsilon_0$. This means that
$$E\left[\int_0^T\left(\int_0^t(f^n-f^m)dW_s\right)^2dt\right] \geqslant \varepsilon^2_0.$$
Thus, when $n,m>N_0$, it follows from Lemma \ref{itoi} that
\begin{align*}
\|x_n-x_m\|^2_c&=E\left[\int_0^T\left(\int_0^t(f^n-f^m)dW_s+\int_0^t(g^n-g^m)ds\right)^2dt\right]\\
&=E\left[\left(\int_0^T\left(\int_0^t(f^n-f^m)ds\right)^2+\left(\int_0^t(g^n-g^m)ds\right)^2dt\right)\right]\geqslant \varepsilon^2_0,
\end{align*}
which is a contradiction with the fact $\|x_n-x_m\|_c<\varepsilon$.  Therefore, we know that there exists an $a\in \overline{A}$ such that $a_n \to a$. Similarly,   there exists a $b\in \overline{B}$ such that $b_n\to b$. Thus,  we have $a+b=x\in \overline{A}+\overline{B}$ and so $ \overline{A+B} \subset \overline{A}+\overline{B}$.

Combining Steps 1 and 2, we obtain $\overline{A}+\overline{B}= \overline{A+B}$.
\end{proof}

\begin{remark}\label{r3.1}
Let $X$ be a Banach space.  Then $\overline{A}+\overline{B} \subset \overline{A+B}$ always holds for any $ A, B \subset X$.
Especially, $\overline{{a}}+\overline{B}=\overline{{a}+B}$ holds for any $a\in X$ and $B\subset X$.
\end{remark}

In order to obtain the set-valued It\^{o}'s formula, we need the following assumption.
\begin{assumption}\label{am3.1}
Assume that $Y_1,Y_2$ are two topological spaces and $\varphi: Y_1\rightarrow Y_2$ is a continuous mapping such that, for any bounded subset $A$ in $Y_1$,
$$\varphi(cl_{Y_1}(A))=cl_{Y_2}(\phi(A)),$$
where $\varphi(A)=\cup_{a\in A}\varphi(a)$.
\end{assumption}

\begin{remark}
\begin{enumerate}[($\romannumeral1$)]
\item If $\varphi: \mathbb{R}^n\to \mathbb{R}^n$ is continuous,  then it is easy to check that $\varphi$ satisfies Assumption \ref{am3.1};
\item Any continuous closed mapping from a Banach space to another one satisfies Assumption \ref{am3.1};
\item The mapping $\varphi: L^2_{ad}([0,T] \times \Omega ,\mathbb{R}^n)\to L^2_{ad}([0,T] \times \Omega ,\mathbb{R}^n)$ defined by $$\varphi(x)=(x^2_1,x^2_2,\cdots,x^2_n), \quad \forall x=(x_1,x_2,\cdots,x_n)\in L^2_{ad}([0,T] \times \Omega ,\mathbb{R}^n)$$
satisfies Assumption \ref{am3.1}.
\end{enumerate}
\end{remark}

Now we are able to derive the set-valued It\^{o}'s formula.
\begin{thm}\label{Ito}
Let $X_t$ be the set-valued It\^{o} process defined by \eqref{e3.1}. If $\phi(t,x)$ is a continuous function satisfying Assumption \ref{am3.1} with continuous partial derivatives $\frac{\partial \phi}{\partial t}$, $\frac{\partial \phi}{\partial x}$ and $\frac{\partial^2 \phi}{\partial x^2}$ for every $x \in X_t^{i}$, then $\phi(t,X_t)$ is a set-valued It\^{o} process and
\begin{align*}
\phi(t,X(t))_k=&\phi(0,x_{0})_k+\int_0^t\sum\limits^n_{i=1}\frac{\partial \phi_k}{\partial x_i}(s,X(s))f_i(s)dW_s\\
&+\int_0^t \left[\frac{\partial \phi_k}{\partial t}(s,X(s))+\sum\limits^n_{i=1} \frac{\partial \phi_k}{\partial x_i}(s,X(s))g_i(s)+\frac{1}{2} \sum\limits_{i=1}^n \frac{\partial^2 \phi_k}{\partial x^2_i}(s,X(s))f^2_i\right]ds,\\
& k=1,\cdots,n \quad \quad \mathbb{P}\mbox{-}a.s.,
\end{align*}
where $\phi(t,x): T\times L^2_{ad}([0,T] \times \Omega ,\mathbb{R}^n)\rightarrow  L^2_{ad}([0,T] \times \Omega ,\mathbb{R}^n)$,
$\phi(t,X(t))=\left\{\phi(t,x(t)):x(t)\in X(t)\right\}$,
$\frac{\partial \phi_k}{\partial t}(s,X(t))=\left\{\frac{\partial \phi_k}{\partial t}(s,x(t)): x(t)\in X(t)\right\}$,
$\frac{\partial \phi_k}{\partial x_i}(s,X(t))=\left\{\frac{\partial \phi_k}{\partial x_i}(s,x(t)):x(t)\in X(t)\right\}$,
$\frac{\partial^2 \phi_k}{\partial x^2_i}(s,X(t))=\left\{\frac{\partial^2 \phi_k}{\partial x^2_i}(s,x(t)):x(t)\in X(t)\right\}$ and $f^2_i=\left\{\tilde{f}_i^T\tilde{f}_i: \tilde{f}_i\in f_i\right\}$.
\end{thm}

\begin{proof}
Since $f(t)$ and $g(t)$ are convex-valued, square integrably bounded, by Lemma \ref{7}, there exist two sequences $\left\{f^{d_1}(s)\right\}_{d_1=1}^{\infty}\subset S^2_{\mathbb{F}}(f)$ and $\left\{g^{d_2}(s)\right\}_{d_2=1}^{\infty}\subset S^2_{\mathbb{F}}(g)$ such that
\begin{align}\label{f}
\int_0^tf(s)dW_s=cl_{\mathbb{L}}\left\{\left(\int_0^tf^{d_1}(s)dW_s\right)(\omega):d_1\geqslant 1\right\} \; \mbox{a.e.}\; \quad \quad \mathbb{P}\mbox{-}a.s..
\end{align}
and
\begin{align*}
\int_0^tg(s)ds=cl_{\mathbb{L}}\left\{\left(\int_0^tg^{d_2}(s)ds\right)(\omega):d_2\geqslant 1\right\} \; \mbox{a.e.}\; \quad \quad \mathbb{P}\mbox{-}a.s..
\end{align*}
It follows from \eqref{e3.1} and Lemma \ref{lm3.3} that
$$X(t)=x_{0}+cl_{\mathbb{L}}\left\{\left(\int_0^tf^{d_1}(s)dW_s\right)(\omega)+\left(\int_0^tg^{d_2}(s)ds\right)(\omega):d_1, d_2 \geqslant 1\right\} \mbox{a.e.}\; \quad \quad \mathbb{P}\mbox{-}a.s..$$
For any given $d_1, d_2\geqslant 1$, let
$$x^{d_1d_2}(t)=x_{0}+\int_0^tf^{d_1}(s)dW_s+\int_0^tg^{d_2}(s)ds \; \mbox{a.e.}\; \quad \quad \mathbb{P}\mbox{-}a.s..$$
Then $x^{d_1d_2}(t)$ is a single-valued It\^{o}'s process and so the classical It\^{o}'s formula shows that
\begin{align}\label{e3.3}
&\quad\;\phi(t,x^{d_1d_2}(t))_k\nonumber\\
&=\phi(0,x_{0})_k+\int_0^t\sum\limits^n_{i=1}\frac{\partial \phi_k}{\partial x_i}(s,x^{d_1d_2}(s))f^{d_1}_i(s)dW_s\nonumber\\
&\quad \mbox{} +\int_0^t \left[\frac{\partial \phi_k}{\partial t}(s,x^{d_1d_2}(s))+\sum\limits^n_{i=1} \frac{\partial \phi_k}{\partial x_i}(s,x^{d_1d_2}(s))g^{d_2}_i(s)+\frac{1}{2} \sum\limits_{i=1}^n \frac{\partial^2 \phi_k}{\partial x^2_i}(s,x^{d_1d_2}(s))(f^{d_1}_i)^2\right]ds
\end{align}
for $k=1,\cdots,n$ and $\mathbb{P}\mbox{-}a.s.$. It follows from Lemma \ref{lm3.3} and Remark \ref{r3.1} that
\begin{align*}
cl_{\mathbb{L}}\left\{x^{d_1d_2}(t): d_1\geqslant 1, d_2\geqslant 1\right\}
=&x_{0}+cl_{\mathbb{L}}\left\{\int_0^tf^{d_1}(s)dW_s+\int_0^tg^{d_2}(s)ds: d_1 \geqslant 1, d_2 \geqslant 1\right\}\\
=&x_{0}+cl_{\mathbb{L}}\left\{\int_0^tf^{d_1}(s)dW_s: d_1 \geqslant 1,\right\}+cl_{\mathbb{L}}\left\{\int_0^tg^{d_2}(s)ds:d_2 \geqslant 1\right\}\\
=&X.
\end{align*}
Let $A=\left\{x^{d_1d_2}(t): d_1\geqslant 1, d_2\geqslant 1\right\}$. Then $\overline{A}=cl_{\mathbb{L}}\left\{x^{d_1d_2}(t): d_1\geqslant 1, d_2\geqslant 1\right\}$
and so Assumption \ref{am3.1} yields
\begin{align}\label{a}
\phi(t,X(t))=\phi(t,\overline{A})=\overline{\phi(t,A)}.
\end{align}
From \eqref{e3.3}, Lemma \ref{lm3.3} and Remark \ref{r3.1}, one has
\begin{align}\label{e3.4}
&\quad\;\phi(t,X(t))_k\nonumber\\
&=x_{0}+cl_{\mathbb{L}} \left\{\int_0^t\sum\limits^n_{i=1}\frac{\partial \phi_k}{\partial x_i}(s,x^{d_1d_2}(s))f^{d_1}_i(s)dW_s: d_1, d_2 \geqslant 1\right\}\nonumber\\
&+cl_{\mathbb{L}} \left\{ \int_0^t \left[\frac{\partial \phi_k}{\partial t}(s,x^{d_1d_2}(s))+\sum\limits^n_{i=1} \frac{\partial \phi_k}{\partial x_i}(s,x^{d_1d_2}(s))g^{d_2}_i(s)+\frac{1}{2} \sum\limits_{i=1}^n \frac{\partial^2 \phi_k}{\partial x^2_i}(s,x^{d_1d_2}(s))(f^{d_1}_i)^2\right]ds:d_1, d_2 \geqslant 1\right\}.
\end{align}
For fixed $f^{d_1}\in S^2_{\mathbb{F}}(f)$, by Lemma \ref{7}, there exists a sequence $\left\{x^{d_3}\right\}\subset S^2_{\mathbb{F}}(X)$ such that
\begin{align}\label{e3.5}
\int_0^t\sum\limits^n_{i=1}\frac{\partial \phi_k}{\partial x_i}(s,X(s))f^{d_1}_i(s)dW_s=cl_{\mathbb{L}}\left\{\int_0^t\sum\limits^n_{i=1}\frac{\partial \phi_k}{\partial x_i}(s,x^{d_3}(s))f^{d_1}_i(s)dW_s, d_3 \geqslant 1\right\}.
\end{align}
For fixed $x^{d_3}\in S^2_{\mathbb{F}}(X)$, it follows from Lemma \ref{7} and \eqref{f} that there exists a sequence $\left\{f^{d_1}\right\}\subset S^2_{\mathbb{F}}(f)$ such that
\begin{align}\label{e3.6}
\int_0^t\sum\limits^n_{i=1}\frac{\partial \phi_k}{\partial x_i}(s,x^{d_3})f_i(s)dW_s=cl_{\mathbb{L}}\left\{\int_0^t\sum\limits^n_{i=1}\frac{\partial \phi_k}{\partial x_i}(s,x^{d_3}(s))f^{d_1}_i(s)dW_s, d_1 \geqslant 1\right\}.
\end{align}
Combining \eqref{e3.5} and \eqref{e3.6}, we have
\begin{align}\label{e3.7}
\int_0^t\sum\limits^n_{i=1}\frac{\partial \phi_k}{\partial x_i}(s,X)f_i(s)dW_s=cl_{\mathbb{L}}\left\{\int_0^t\sum\limits^n_{i=1}\frac{\partial \phi_k}{\partial x_i}(s,x^{d_3}(s))f^{d_1}_i(s)dW_s, d_1,d_3 \geqslant 1\right\}.
\end{align}
Let
$$M=\left\{\int_0^t\sum\limits^n_{i=1}\frac{\partial \phi_k}{\partial x_i}(s,x^{d_3}(s))f^{d_1}_i(s)dW_s, d_1,d_3 \geqslant 1\right\}
$$
and
$$N=\left\{\int_0^t\sum\limits^n_{i=1}\frac{\partial \phi_k}{\partial x_i}(s,x^{d_1d_2}(s))f^{d_1}_i(s)dW_s, d_1,d_2 \geqslant 1\right\}.$$
Since $x^{d_3}\in S^2_{\mathbb{F}}(X)=S^2_{\mathbb{F}}(cl_{\mathbb{L}}{A})$, we can choose $d_3=d_1d_2$ and so $cl_{\mathbb{L}}{N}\subset cl_{\mathbb{L}}{M}$. On the other hand, for any $a_0 \in cl_{\mathbb{L}}{M}$, there exists a sequence  $\left\{a_q\right\}\subset M$ such that $a_q\to a_0$. Since $a_q\in M$, there exist $x_q\in \left\{x^d_3\right\}_{d_3=1}^\infty \subset S^2_{\mathbb{F}}(cl_{\mathbb{L}}{A})$ and $f_q\in \left\{f^d_1\right\}_{d_1=1}^\infty \subset S^2_{\mathbb{F}}(f)$ such that
$$a_q=\int_0^t\sum\limits^n_{i=1}\frac{\partial \phi_k}{\partial x_i}(s,x_q(s))(f_q)_i(s)dW_s.$$
For fixed $x_q\in S^2_{\mathbb{F}}(cl_{\mathbb{L}}{A})$, there exists a sequence  $\left\{x_{ql}\right\}\subset A$ such that $x_{ql}\to x_q$. Letting $f_{ql}=f_q$, we have
$$\lim\limits_{ql\rightarrow +\infty}\int_0^t\sum\limits^n_{i=1}\frac{\partial \phi_k}{\partial x_i}(s,x_{ql}(s))(f_{ql})_i(s)dW_s=a_0,$$
which implies $cl_{\mathbb{L}}{M}\subset cl_{\mathbb{L}}{N}$. Thus, we have $cl_{\mathbb{L}}{M}=cl_{\mathbb{L}}{N}$ and so
\begin{align}\label{e3.8}
\int_0^t\sum\limits^n_{i=1}\frac{\partial \phi_k}{\partial x_i}(s,X)f_i(s)dW_s=cl_{\mathbb{L}}\left\{\int_0^t\sum\limits^n_{i=1}\frac{\partial \phi_k}{\partial x_i}(s,x^{d_1d_2}(s))f^{d_1}_i(s)dW_s, d_1,d_2 \geqslant 1\right\}.
\end{align}

Similarly, we can obtain the following equality
\begin{align}\label{e3.9}
&\int_0^t \left[\frac{\partial \phi_k}{\partial t}(s,X(s))+\sum\limits^n_{i=1} \frac{\partial \phi_k}{\partial x_i}(s,X(s))g_i(s)+\frac{1}{2} \sum\limits_{i=1}^n \frac{\partial^2 \phi_k}{\partial x^2_i}(s,X(s))f^2_i\right]ds\nonumber\\
&=cl_{\mathbb{L}}\left\{ \int_0^t \left[\frac{\partial \phi_k}{\partial t}(s,x^{d_1d_2})+\sum\limits^n_{i=1} \frac{\partial \phi_k}{\partial x_i}(s,x^{d_1d_2})(g^{d_2})_i(s)+\frac{1}{2} \sum\limits_{i=1}^n \frac{\partial^2 \phi_k}{\partial x^2_i}(s,x^{d_1d_2})(f^{d_1}_i)^2\right]ds,d_1,d_2\geqslant 1 \right\}.
\end{align}
Combining \eqref{a}, \eqref{e3.8}, \eqref{e3.9}, we have
\begin{align*}
\phi(t,X(t))_k
=&\phi(0,X(0))_k+\int_0^t\sum\limits^n_i\frac{\partial \phi_k}{\partial x_i}(s,X(s))f_i(s)dW_s\\
&+\int_0^t \left[\frac{\partial \phi_k}{\partial t}(s,X(s))+\sum\limits^n_i \frac{\partial \phi_k}{\partial x_i}(s,X(s))g_i(s)+\frac{1}{2} \sum\limits_{i=1}^n \frac{\partial^2 \phi_k}{\partial x^2_i}(s,X(s))f^2_i\right]ds
  \quad \mathbb{P}\mbox{-}a.s.
\end{align*}
for $k=1,\cdots,n$. This completes the proof.
\end{proof}

\begin{remark}
Theorem \ref{Ito} is a set-valued version of the classical It\^{o}'s formula.
\end{remark}

\begin{thm}\label{cor3.1}
Let $X(t)=x_T+\int_t^Tf(s,Z(s))ds\ominus \int_t^TZ(s)dW_s$ and $Z\in\mathbb{G}$.
Then
\begin{align*}
X^2_k+\int_t^TZ_k^2(s)ds\subset (x^2_T)_k+2\int_t^TX_kf_k(s,Z(s))ds-2\int_t^TX_kZ_k(s)dW_s,\, k=1,\cdots,n \quad \mathbb{P}\mbox{-}a.s..
\end{align*}
\end{thm}
\begin{proof}
From Lemma \ref{7} and the property of Hukuhara difference, there exist $\left\{z^{d_1}\right\}\in S^2_{\mathbb{F}}(Z)$ such that
\begin{align*}
\int_t^Tf(s,Z)ds\ominus \int_t^TZ_sdW_s\subset \int_t^Tf(s,Z(s))ds- \int_t^TZ(s)dW_s
=cl_{\mathbb{L}}\left\{\int_t^Tf(s,z^{d_1})ds-\int_t^Tz^{d_1}dW_s: d_1 \geqslant 1\right\}.
\end{align*}
Let
$$X^*(t)=x_T+\int_t^Tf(s,Z(s))ds- \int_t^TZ(s)dW_s,$$
$$x^{d_1}=x_T+\int_t^Tf(s,z^{d_1})ds-\int_t^Tz^{d_1}dW_s.$$
Then $cl_{\mathbb{L}}\left\{x^{d_1}:d_1 \geqslant 1\right\}=X^*$ and $x^{d_1}(t)$ is a single-valued It\^{o}'s process. Let $\phi(t,x)=(x^2_1,x^2_2,\cdots, x^2_n)$.
By the classical It\^{o}'s formula, we have
\begin{align*}
(x^{d_1}_k)^2=(x^2_T)_k+2\int_t^Tx^{d_1}_kf_k(s,z^{d_1})-(z^{d_1}_k)^2(s)ds-2\int_t^Tx^{d_1}z^{d_1}_k(s)dW_s.
\end{align*}
Similar to  the proof of Theorem \ref{Ito}, we can obtain
\begin{align*}
&{X^*}^2_k(t)=cl_{\mathbb{L}}\left\{(x^{d_1}_k)^2:d_1\geqslant 1 \right\}\\
=&(x^2_T)_k+cl_{\mathbb{L}}\left\{2\int_t^Tx^{d_1}_kf_k(s,z^{d_1})-(z^{d_1}_k)^2(s)ds:d_1\geqslant 1 \right\}+cl_{\mathbb{L}}\left\{-2\int_t^Tx^{d_1}z^{d_1}_k(s)dW_s:d_1\geqslant 1 \right\}
\end{align*}
and so
\begin{align*}
X^2_k(t)\subset {X^*}^2_k= x^2_T+2\int_t^TX_kf_k(s,Z(s))-Z_k^2(s)ds-2\int_t^TX_kZ_k(s)dW_s,\, k=1,\cdots,n \quad \mathbb{P}\mbox{-}a.s..
\end{align*}
For any given $y\in X^2_k+\int_t^TZ_k^2(s)ds$, there exist $m\in X^2_k$ and $n\in \int_t^TZ_k^2(s)ds$ such that $y=m+n$. Since $m\in X^2_k\subset cl_{\mathbb{L}}\left\{(x^{d_1}_k)^2:d_1\geqslant 1 \right\}$ and $n\in \int_t^TZ_k^2(s)ds=cl_{\mathbb{L}}\left\{\int_t^T(z^{d_1}_k)^2dW_s: d_1 \geqslant 1\right\},$ we know that there exist $\left\{m_j\right\}\subset \left\{(x^{d_1}_k)^2:d_1\geqslant 1 \right\}$ and $\left\{n_j\right\}\subset \left\{\int_t^T(z^{d_1}_k)^2dW_s: d_1 \geqslant 1\right\}$ such that $m_j+n_j\to m+n$. It follows Lemma \ref{lm3.3} and Remark \ref{r3.1} that
\begin{align*}
m+n&\in cl_{\mathbb{L}}\left\{(x^2_T)_k+2\int_t^Tx^{d_1}_kf_k(s,z^{d_1})ds-2\int_t^Tx^{d_1}z^{d_1}_k(s)dW_s:d_1\geqslant 1 \right\}\\
&=(x^2_T)_k+cl_{\mathbb{L}}\left\{2\int_t^Tx^{d_1}_kf_k(s,z^{d_1})ds:d_1\geqslant 1 \right\}+cl_{\mathbb{L}}\left\{2\int_t^Tx^{d_1}z^{d_1}_k(s)dW_s:d_1\geqslant 1 \right\}\\
&=(x^2_T)_k+2\int_t^TX_kf_k(s,Z(s))ds-2\int_t^TX_kZ_k(s)dW_s.
\end{align*}
This shows that
\begin{align*}
X^2_k+\int_t^TZ_k^2(s)ds\subset (x^2_T)_k+2\int_t^TX_kf_k(s,Z(s))ds-2\int_t^TX_kZ_k(s)dW_s,\, k=1,\cdots,n \quad \quad \mathbb{P}\mbox{-}a.s.
\end{align*}
and so the proof is completed.
\end{proof}
Similar as the proof of Theorem \ref{cor3.1}, we can get the following Corollary.
\begin{cor}\label{cor3.2}
Let $X_t=x_T+\int_t^Tf(s,X_s,Z(s))ds\ominus \int_t^TZ(s)dW_s$ and $Z\in\mathbb{G}$.
Then
\begin{align*}
(X^2_t)_k+\int_t^TZ_k^2(s)ds\subset (x^2_T)_k+2\int_t^T(X_s)_kf_k(s,X_s,Z(s))ds-2\int_t^T(X_s)_kZ_k(s)dW_s,\, k=1,\cdots,n, \quad \mathbb{P}\mbox{-}a.s..
\end{align*}
\end{cor}

\section{An application to GSVBSDE}
\noindent \setcounter{equation}{0}

In this section, we apply the set-valued It\^{o}'s formula to obtain the existence and uniqueness of solutions for GSVBSDE \eqref{e1.1}. To this end, we first show the following lemma, which is a set-valued version of the It\^{o} isometry.

\begin{lm}\label{sitoi}
Suppose $f\in \mathcal{L}^2_{ad}([0,T]\times \Omega,\mathcal{P}(\mathbb{R}^m))$. Then the integral $I(f)=\int_0^T f(t)dW(t)$ is a set-valued random variable with $\mathbb{E}[I(f)]=0$ and
\begin{align}\label{sitoi-0}
\mathbb{E}\left[I(t)\right]^2= \mathbb{E} \left[\int_0^Tf^2(t)dt\right].
\end{align}
\end{lm}
\begin{proof}
By Lemma \ref{7}, there exists a sequence $\left\{f^d(t)\right\}_{d=1}^{\infty} \subset S_{\mathbb{F}}^2(f)$ such that
\begin{align}\label{sitoi-1}
\int_0^T f(s) dW_s=cl_{\mathbb{L}}\left\{\int_0^T f^d(s) dW_s:d\geqslant 1\right\} \quad \mathbb{P}\mbox{-}a.s.
\end{align}
and
\begin{align}\label{sitoi-2}
\int_0^T f^2(s) ds=cl_{\mathbb{L}}\left\{\int_0^T (f^d(s))^2 ds:d\geqslant 1\right\} \quad \mathbb{P}\mbox{-}a.s..
\end{align}
It follows from \eqref{sitoi-1} that
$$
\mathbb{E}\left[\int_0^T f(s) dW_s\right]=
\mathbb{E}\left[cl_{\mathbb{L}}\left\{\int_0^T f^d(s) dW_s:d\geqslant 1\right\}\right].
$$
For any given $a_0\in cl_{\mathbb{L}}\left\{\int_0^T f^d(s) dW_s:d\geqslant 1\right\}$, there exists a sequence $\left\{f^n(t)\right\}_{n=1}^{\infty} \subset \left\{f^d(t)\right\}_{d=1}^{\infty}$ such that $\int_0^Tf^n(s)ds \to a_0$. Thus, by Lemma \ref{itoi}, we know that $\mathbb{E}\int_0^Tf^n(s)ds=0\to \mathbb{E}a_0$.
This implies that $\mathbb{E}a_0=0$ and so
$$\mathbb{E}[I(f)]=\mathbb{E}\left[\int_0^T f(s) dW_s\right]=
\mathbb{E}\left[cl_{\mathbb{L}}\left\{\int_0^T f^d(s) dW_s:d\geqslant 1\right\}\right]=0.$$

Next we show that \eqref{sitoi-0} holds. In fact, by \eqref{sitoi-1}, we have
$$\mathbb{E}[I(f)]^2=\mathbb{E}\left[cl^2_{\mathbb{L}}\left\{\int_0^T f^d(s) dW_s:d\geqslant 1\right\}\right] \quad \mathbb{P}\mbox{-}a.s.,$$
where
$$cl^2_{\mathbb{L}}\left\{\int_0^T f^d(s) dW_s:d\geqslant 1\right\}=\left\{x^2: x\in cl_{\mathbb{L}}\left\{\int_0^T f^d(s) dW_s:d\geqslant 1\right\}\right\}.$$
For any given $b_0\in cl_{\mathbb{L}}\left\{\int_0^T f^d(s) dW_s:d\geqslant 1\right\}$, there exists a sequence $\left\{f^k(t)\right\}_{k=1}^{\infty} \subset \left\{f^d(t)\right\}_{d=1}^{\infty}$ such $\int_0^Tf^k(s)dW_s \to b_0$.
It follows from Lemma \ref{itoi} that $$\mathbb{E}\left[\left(\int_0^Tf^k(s)dW_s\right)^2\right]=\mathbb{E}\left[\int_0^T(f^k(s))^2ds\right] \to \mathbb{E}[b^2_0].$$
This leads to
$$\mathbb{E}[I(f)]^2=\mathbb{E}\left[cl_{\mathbb{L}}\left\{\int_0^T (f^d(s))^2 ds:d\geqslant 1\right\}\right].$$
Thus, by \eqref{sitoi-2}, we know that  \eqref{sitoi-0} holds.
\end{proof}

Recall that, for $A\in \mathcal{K}(\mathbb{X})$ and $ a\in A$, $a$ is called an extreme point of $A$ if it can not be written as a strict convex combination of two points in $A$, that is for every $y_1, y_2 \in A$ and $\lambda \in (0,1)$, we have $a\neq \lambda y_1+ (1-\lambda)y_2$. Then, we denote $ext(A)$ to be the set of all extreme points of $A$.

\begin{lm} \cite{Ararat2020Set-valued} \label{hd}
Let $A,B\in \mathcal{K}(\mathbb{X})$. The Hukuhara differential $A\ominus B$ exists if and only if for every $a\in ext(A)$, there exists $x\in \mathbb{X}$ such that $a\in x+B$ and $ x+B\subseteq A$. In this case, $A\ominus B$ is unique, closed, convex, and
$$A\ominus B=\left\{x\in \mathbb{X}|x+B \subseteq A\right\}.$$
\end{lm}

\begin{lm} \label{equal}
Let $X(t), Y(t) \in \mathcal{L}^2_{ad}([0,T] \times \Omega, \mathcal{L}(\mathbb{R}^{n\times m}))$ such that
\begin{align} \label{eq}
\int_0^t X(s) dW_s =\int_0^t Y(s)dW_s, \: \forall t\in[0,T].
\end{align}
Then $X(t)\ominus Y(t)$ exists and $X(t)=Y(t)$ $\mathbb{P}$-a.s..
\end{lm}
\begin{proof}
Assume that $X(t)\ominus Y(t)$ does not exist. Then Lemma \ref{hd} implies that there exists $a_0(t) \in ext(X(t))$ such that, for any $x(t) \in L^2_{ad}([0,T] \times \Omega ,\mathbb{R}^{n\times m})$,  $a_0(t) \notin x(t)+Y(t)$ or $x(t)+Y(t) \nsubseteq X(t)$. Taking $x(t)=0$, one has  $a_0(t) \notin Y(t)$ or $Y(t) \nsubseteq X(t)$.
Suppose that $a_0(t) \notin Y(t)$ holds. Then it follows from $a_0(t) \in ext(X(t))$ and \eqref{eq} that
$$
\int_0^t a_0(s) dW_s \in \int_0^t X(s) dW_s=\int_0^t Y(s)dW_s, \: \forall t\in[0,T].
$$
From Lemma \ref{7}, there exists a sequence $\left\{y^d(t)\right\}_{d=1}^{\infty} \subset S_{\mathbb{F}}^2(Y(t))$ such that
\begin{align}\label{equal-1}
\int_0^T Y(s) dW_s=cl_{\mathbb{L}}\left\{\int_0^T y^d(s) dW_s:d\geqslant 1\right\} \quad \mathbb{P}\mbox{-}a.s.
\end{align}
This leads that there exists a sequence$\left\{y^n(t)\right\}_{n=1}^{\infty} \subset \left\{y^d(t)\right\}_{d=1}^{\infty}$ such that  $\int_0^t (y^n - a_0)dW_s \to 0$. It follows from Lemma \ref{itoi} that $\mathbb{E}[\int_0^T (y^n - a_0)^2ds]= \| y^n(s) - a_0 \|^2_c \to 0$, which implies $y^n(s) \to a_0$ $\mathbb{P}$-a.s.. Since $S_{\mathbb{F}}^2(Y(t))$ is closed, we have $a_0 \in S_{\mathbb{F}}^2(Y(t))\subset Y(t)$.
This in contradiction with $a_0(t) \notin Y(t)$. Similarly, when $Y(t) \nsubseteq X(t)$, we can obtain a contradiction.  Thus, we know that $X(t)\ominus Y(t)$ exists and so   \eqref{eq} implies that
$$\int_0^t X(s)\ominus  Y(s)dW_s=0, \: \forall t\in[0,T].$$
By Lemma \ref{sitoi}, one has
$$\mathbb{E}\left[\int_0^t (X(s)\ominus  Y(s))^2ds\right]=0, \: \forall t\in[0,T].$$

Since $X(t),Y(t) \in \mathcal{L}^2_{ad}([0,T] \times \Omega, \mathcal{L}(\mathbb{R}^{n\times m}))$, let $x^0(t)\in X(t)\ominus  Y(t)$ such that $x_0(t)=\|X(t)\ominus  Y(t)\|$, by taking $t=T$, we have
$$\mathbb{E}\left[\int_0^T \|X(s)\ominus  Y(s)\|^2ds\right]=\mathbb{E}\left[\int_0^T x^2_0(t)ds\right]\in \mathbb{E}\left[\int_0^T (X(s)\ominus  Y(s))^2ds\right]=0,$$
which implies that
$$\mathbb{E}\left[\int_0^T \|X(s)\ominus  Y(s)\|^2ds\right]=\|X(s)\ominus  Y(s)\|_s=0$$
and so $X(t)=Y(t)$ $\mathbb{P}$-a.s..
\end{proof}

We also need the following lemma.
\begin{lm}\cite{Ararat2020Set-valued}\label{l4.1}
Let $\xi \in \mathbb{L}_T^2(\Omega,\mathcal{K}(\mathbb{R}^n))$ and $f(t,\omega)\in \mathcal{L}^2_{ad}([0,T]\times \Omega,\mathcal{K}(\mathbb{R}^n))$. Then there exists a unique pair $(X,Z)\in \mathcal{L}^2_{ad}([0,T]\times \Omega,\mathcal{K}(\mathbb{R}^n))\times \mathbb{G}$ such that
$$X(t)=\xi +\int_t^Tf(s,\omega)ds \ominus \int_t^TZ(s)dW_s,  \quad \quad \mathbb{P}\mbox{-}a.s.\quad t\in [0.T].$$
\end{lm}

We first consider the following set-valued equation
$$X(t)+\int_t^TZ(s)dW_s=\xi +\int_t^Tf(s,Z(s))ds,  \quad \quad \mathbb{P}\mbox{-}a.s.\quad t\in [0.T],$$
where $\xi \in \mathbb{L}_T^2(\Omega,\mathcal{K}(\mathbb{R}^n))$, $Z_t \in \mathbb{G}$  and $f: [0,T] \times \Omega  \times \mathcal{P}(\mathbb{R}^{n\times m}) \rightarrow \mathcal{K}(\mathbb{R}^n)$.
\begin{assumption}\label{c51}
Assume that $f: [0,T] \times \Omega  \times \mathcal{P}(\mathbb{R}^{n\times m}) \rightarrow \mathcal{K}(\mathbb{R}^n)$ satisfies the following conditions:
\begin{enumerate}[($\romannumeral1$)]
\item for any given $A\in \mathcal{P}(\mathbb{R}^{n\times m})$, $f(\cdot,\cdot,A) \in \mathcal{L}^2_{ad}([0,T]\times \Omega,\mathcal{K}(\mathbb{R}^n))$;
\item for any fixed $(t,\omega) \in [0,T]\times \Omega$ and $A,B\in \mathcal{P}(\mathbb{R}^{n\times m})$, the Hukuhara difference
$f(t,\omega,A)\ominus f(t,\omega,B)$ exists whenever $A\ominus B$ exists;
\item for any given $A,B \in \mathcal{P}(\mathbb{R}^{n\times m})$ with $A\ominus B$ existing, there is a constant $c>0$ such that
$$\|f(t,\omega,A)\ominus f(t,\omega,B)\| \leqslant c\|A \ominus B\|, \quad \forall (t,\omega) \in [0,T]\times \Omega.$$
\end{enumerate}

\end{assumption}

\begin{thm}\label{t52}
Let $\xi \in \mathbb{L}_T^2(\Omega,\mathcal{K}(\mathbb{R}^n))$ and $f(t,\omega,Z(t))$ satisfy Assumption \ref{c51}. Then there exists a pair $(X,Z)\in \mathcal{L}^2_{ad}([0,T]\times \Omega,\mathcal{K}(\mathbb{R}^n))\times \mathbb{G}$ such that
\begin{align}\label{p1}
X(t)+\int_t^TZ(s)dW_s=\xi +\int_t^Tf(s,\omega,Z(s))ds,  \quad \forall t\in [0,T], \quad \mathbb{P}\mbox{-}a.s..
\end{align}
Moreover, if $(X_1(t),Z_1(t))$ and $(X_2(t),Z_2(t))$ are two solutions of \eqref{p1} and $Z_1(t)\ominus Z_2(t)$ exists, then $X_1(t)=X_2(t)$ and $Z_1(t)=Z_2(t)$ $\mathbb{P}$-a.s..
\end{thm}
\begin{proof}
We first show the existence of solutions for \eqref{p1}. Let $X^0=\left\{0\right\}$ and $Z^0=\left\{0\right\}$.  By Assumption \ref{c51}(i), we know that $f(t,Z^0)\in \mathcal{L}^2_{ad}([0,T]\times \Omega,\mathcal{K}(\mathbb{R}^n))$. It follows from Lemma \ref{l4.1} that there is a pair $(X^1(t), Z^1(t))\in \mathcal{L}^2_{ad}([0,T]\times \Omega,\mathcal{K}(\mathbb{R}^n))\times \mathbb{G}$ such that
$$
X^1(t)+\int_t^TZ^1(s)dW_s=\xi +\int_t^Tf(s,\omega,Z^0(s))ds,  \quad \forall t\in [0,T], \quad \mathbb{P}\mbox{-}a.s..
$$
In the same way, we can obtain a sequence $\{(X^p(t),Z^p(t))\}\subset \mathcal{L}^2_{ad}([0,T]\times \Omega,\mathcal{K}(\mathbb{R}^n))\times \mathbb{G}$ such that
\begin{equation}\label{e4+1}
X^p(t)+\int_t^TZ^p(s)dW_s=\xi +\int_t^Tf(s,\omega,Z^{p-1}(s))ds, \quad \forall t\in [0,T], \quad \mathbb{P}\mbox{-}a.s., \quad p=1,2,\cdots.
\end{equation}

Now we conclude that for each $t\in [0,T]$, $X^p(t)\ominus X^{p-1}(t)$ and $Z^p(t)\ominus Z^{p-1}(t)$ exist $\mathbb{P}$-a.s..
In fact, for $p=1$, it is clear that $X^1(t)\ominus X^{0}$ and $Z^1(t)\ominus Z^{0}$ exist $\mathbb{P}$-a.s. since $X^0=\left\{0\right\}$ and $Z^0=\left\{0\right\}$. Assume that the assertion is true for $p-1$ ($p>1$). Then $X^{p-1}(t)\ominus X^{p-2}(t)$ and $Z^{p-1}(t)\ominus Z^{p-2}(t)$ exist $\mathbb{P}$-a.s.. It follows from Assumption \ref{c51} (ii) that $f(t,\omega,Z^{p-1}(t))\ominus f(t,\omega,Z^{p-2}(t))$ exists for all $t\in [0,T]$ and $\mathbb{P}$-a.s.. Since $f(t,\omega,Z^{p-1}(t)), f(t,\omega,Z^{p-2}(t)) \in \mathcal{L}^2_{ad}([0,T]\times \Omega,\mathcal{K}(\mathbb{R}^n))$, Lemma \ref{42} shows that $f(t,\omega,Z^{p-1}(t))\ominus f(t,\omega,Z^{p-2}(t)) \in \mathcal{L}^2_{ad}([0,T]\times \Omega,\mathcal{K}(\mathbb{R}^n))$ and
\begin{align*}
&\int_t^Tf(s,\omega,Z^{p-1}(s))ds \ominus \int_t^Tf(s,\omega,Z^{p-2}(s))ds\\
&=\int^T_t\left[f(s,\omega,Z^{p-1}(s))\ominus f(s,\omega,Z^{p-2}(s))\right]ds,  \quad \forall t\in [0,T], \quad \mathbb{P}\mbox{-}a.s..
\end{align*}
For fixed $t\in [0,T]$, we know that $\int_t^T f(s,\omega,Z^{p-1}(s))ds, \; \int_t^T f(s,\omega,Z^{p-2}(s))ds \in \mathbb{L}^2_T(\Omega,\mathcal{L}(\mathbb{R}^n))$ and so  Lemma \ref{5} implies that
\begin{align}\label{e4+2}
&\quad\; \mathbb{E}\left[\int_t^Tf(s,\omega,Z^{p-1}(s))ds\bigg|\mathcal{F}_t\right] \ominus \mathbb{E}\left[\int_t^Tf(s,\omega,Z^{p-2}(s))ds\bigg|\mathcal{F}_t\right]\nonumber\\
&=\mathbb{E}\left[\int_t^Tf(s,\omega,Z^{p-1}(s))\ominus f(s,\omega,Z^{n-2}(s))ds\bigg|\mathcal{F}_t\right].
\end{align}
By \eqref{e4+1}, we have
$$
\mathbb{E}\left[X^p(t)+\int_t^TZ^p(s)dW_s\bigg|\mathcal{F}_t\right]=\mathbb{E}\left[\xi +\int_t^Tf(s,\omega,Z^{p-1}(s))ds\bigg|\mathcal{F}_t\right].
$$
It follows from Definition \ref{d25} and Lemma \ref{sitoi} that, $p=1,2,\cdots$,
\begin{align}\label{e4+3}
X^p(t)=X^p(t)+\mathbb{E}\left[\int_t^TZ^p(s)dW_s\right]&=\mathbb{E}\left[\xi +\int_t^Tf(s,\omega,Z^{p-1}(s))ds\bigg|\mathcal{F}_t\right],
\quad \forall t\in [0,T], \quad \mathbb{P}\mbox{-}a.s..
\end{align}
Now from \eqref{e4+2} and \eqref{e4+3}, one has
$$
X^p(t)\ominus X^{p-1}(t)=\mathbb{E}\left[\int_t^T f(s,\omega,Z^{p-1}(s))ds\bigg|\mathcal{F}_t\right]\ominus \mathbb{E}\left[\int_t^T f(s,\omega,Z^{p-2}(s))ds\bigg|\mathcal{F}_t\right], \; \forall t\in [0,T], \; \mathbb{P}\mbox{-}a.s.
$$
with $p=1,2,\cdots$. This means $X^p\ominus X^{p-1}$ exists $\mathbb{P}$-a.s.. Next we show that $Z^p\ominus Z^{p-1}$ exists $\mathbb{P}$-a.s.. To this end, let
$$M^{p}(t)=\mathbb{E}\left[\xi +\int_0^T f(s,\omega,Z^{p-1}(s))ds\bigg|\mathcal{F}_t\right], \quad \forall t\in [0,T],\; p=1,2,\cdots.
$$
Then, by Lemma \ref{12.1} (ii) and Lemma \ref{5}, we have
$$M^{p}(t) \ominus M^{p-1}(t)=\mathbb{E}\left[ \int_0^T[f(s,\omega,Z^{p-1}(s))\ominus f(s,\omega,Z^{p-2}(s))]ds\bigg|\mathcal{F}_t\right], \quad \forall t\in [0,T],\; p=1,2,\cdots.$$
By \eqref{e4+1} with $t=0$, one has
$$
\mathbb{E}\left[X^p(0)+\int_0^TZ^p(s)dW_s\bigg|\mathcal{F}_t\right]=\mathbb{E}\left[\xi +\int_0^T f(s,\omega,Z^{p-1}(s))ds\bigg|\mathcal{F}_t\right]
$$
and so  Definition \ref{d25} shows that
\begin{align}\label{e4+4}
X^p(0)+\int_0^tZ^p(s)dW_s=\mathbb{E}\left[\xi +\int_0^T f(s,\omega,Z^{p-1}(s))ds\bigg|\mathcal{F}_t\right]=M^n(t).
\end{align}
Since $M^{p}(t) \ominus M^{p-1}(t)$ is a uniformly square-integrable set-valued martingale, it follows from Lemma \ref{6} that there exists $\tilde{Z}^p(t) \in \mathbb{G}$ such that
$$M^{p}(t) \ominus M^{p-1}(t)=M^{p}(0) \ominus M^{p-1}(0)+\int_0^t \tilde{Z}^p(s)dW_s=X^{p}(0)\ominus X^{p-1}(0)+\int_0^t\tilde{Z}^p(s)dW_s.$$
Moreover, by \eqref{e4+4} and Lemma \ref{l2.11}, one has
\begin{align*}
X^p(0)+\int_0^t Z^p(s)dW_s&=M^{p-1}(t)+M^p(t)\ominus M^{p-1}(t)\\
&=X^{p-1}(0)+\int_0^t Z^{p-1}(s)dW_s +X^{n}(0)\ominus X^{p-1}(0)+\int_0^t \tilde{Z}^p(s)dW_s\\
&=X^{p-1}(0)+X^{p}(0)\ominus X^{p-1}(0)+\int_0^t [Z^{p-1}(s)+\tilde{Z}^p(s)]dW_s, \; \forall t\in [0,T], \; \mathbb{P}\mbox{-}a.s.
\end{align*}
with $p=1,2,\cdots$. From Lemma \ref{equal}, this yields that $Z^p(t)=Z^{p-1}(t)+\tilde{Z}^p(t) \in \mathbb{G}$ and so $Z^p(t)\ominus Z^{p-1}(t)=\tilde{Z}^p(t)$ exists $\mathbb{P}$-a.s.. Thus, by the induction, we know that for each $t\in [0,T]$, $X^p(t)\ominus X^{p-1}(t)$ and $Z^p(t)\ominus Z^{p-1}(t)$ exist $\mathbb{P}$-a.s..

Next we show that there is a pair  $(X(t), Z(t))\in \mathcal{L}^2_{ad}([0,T]\times \Omega,\mathcal{K}(\mathbb{R}^n))\times \mathbb{G}$ such that
$$\|X(s)\ominus X^p(s)\|_s \rightarrow 0,\quad \|Z(s)\ominus Z^p(s)\|_s \rightarrow 0.$$
Indeed, it follows from \eqref{e4+1} that
\begin{align}
&X^{p+1}(t)\ominus X^{p}(t)+\int_t^TZ^{p+1}(s)\ominus Z^p(s)dW_s\nonumber\\
&= \int_t^Tf(s,\omega,Z^{p}(s))\ominus f(s,\omega,Z^{p-1}(s))ds, \; \forall t\in [0,T], \; \mathbb{P}\mbox{-}a.s., \; p=1,2,\cdots,
\end{align}
which is equal to
$$X^{p+1}_i(t)\ominus X^{p}_i(t)+\int_t^TZ^{p+1}_i(s)\ominus Z^p_i(s) dW_s= \int_t^Tf_i(s,Z^{p}(s))\ominus f_i(s,Z^{p-1}(s))ds,\quad t\in[0,T], \quad i=1,\cdots,n,$$
where
$$X^{p+1}_i(t)\ominus X^{p}_i(t), \: f_i(t,Z^{p}(t))\ominus f_i(t,Z^{p-1}(t)) \in \mathcal{L}^2_{ad}([0,T]\times \Omega,\mathcal{K}(\mathbb{R}))$$
and $$(Z^{p+1}_i(t)\ominus Z^p_i(t))^T \\ \in \mathcal{K}_w(\mathcal{L}^2_{ad}([0,T]\times \Omega,\mathcal{K}(\mathbb{R}^m)))$$
are the components of $X^{p+1}(t)\ominus X^{p}(t)$,
$f(s,Z^{p}(s))\ominus f(t,Z^{p-1}(t))$ and $Z^{p+1}(t)\ominus Z^p(t)$, respectively. It follows from Theorem \ref{cor3.1} that
\begin{align*}
&\quad\;(X^{p+1}_i(t) \ominus X^p_i(t))^2 + \int_t^T(Z^{p+1}_i(s) \ominus Z^p_i(s))^2ds \\
&\subset 2\int_t^T(X^{p+1}_i(s)\ominus X^p_i(s))(f_i(s,Z^p(s))\ominus f_i(s,Z^{p-1}(s)))ds - 2\int_t^T(X^{p+1}_i(s)\ominus X^p_i(s))(Z^{p+1}_i(s)\ominus Z^p_i(s))dW_s,\\
&\quad i=1,\cdots,n.
\end{align*}
From Lemma \ref{sitoi}, Assumption \ref{c51} and the basic inequality $\frac{1}{\rho}a^2+\rho b^2\geqslant 2ab$ with $\rho=2c^2$, we have
\begin{align}\label{e4+5}
&\quad\;\mathbb{E}\|X^{p+1}(t)\ominus X^p(t)\|^2+\mathbb{E}\int_t^T\|Z^{p+1}(s) \ominus Z^p(s)\|^2ds\nonumber\\
&\leqslant \frac{1}{2}\mathbb{E}\int_t^T\|Z^p(s) \ominus Z^{p-1}(s)\|^2ds + 2c^2\mathbb{E}\int_t^T\|X^{p+1}(s)\ominus X^{p}(s)\|^2ds.
\end{align}
Denote
$$u_p(t)=\mathbb{E}\int_t^T\|X^{p}(s)\ominus X^{p-1}(s)\|^2ds, \quad v_p(t)=\mathbb{E}\int_t^T\|Z^{p}(s) \ominus Z^{p-1}(s)\|^2ds, \quad p=1,2,\cdots.$$
Then \eqref{e4+5} leads to
\begin{align}\label{e4.1+}
-\frac{d}{dt}\left(u_{p+1}(t)e^{2c^2t}\right)+e^{2c^2t}v_{p+1}(t)\leqslant \frac{1}{2}e^{2c^2t}v_p(t).
\end{align}
Integrating from $t$ to $T$ for two sides of \eqref{e4.1+},  we obtain
\begin{align}\label{e4.2+}
u_{p+1}(t)+\int_t^T e^{2c^2(s-t)}v_{p+1}(s)ds \leqslant \frac{1}{2} \int_t^T e^{2c^2(s-t)}v_p(s)ds.
\end{align}
Noting $u_{p+1}(t),v_{p+1}(t)\geqslant 0$, it follows from \eqref{e4.2+} that
$$\int_t^T e^{2c^2(s-t)}v_{p+1}(s)ds \leqslant \frac{1}{2} \int_t^T e^{2c^2(s-t)}v_p(s)ds.$$
Iterating the inequality and taking $t=0$ in above inequality, we have
$$\int_0^Te^{2c^2t}v_{p+1}(t)dt \leqslant 2^{-p}\bar{c}e^{2c^2},$$
where
$$\bar{c}=\sup\limits_{0\leqslant t\leqslant T}v_1(t)=\mathbb{E}\int_0^T\|Z^1(t)\|dt.$$
Moreover, it follows from \eqref{e4.2+} that
$$
u_{p+1}(t)\leqslant \frac{1}{2} \int_t^T e^{2c^2(s-t)}v_p(s)ds
$$
and so
\begin{align}\label{u1}
u_{p+1}(0) \leqslant 2^{-p}\bar{c}e^{2c^2}.
\end{align}
However, from \eqref{e4.1+}, \eqref{u1} and the fact that $\frac{d}{dt}u_{p+1}(t)\leqslant 0$, we have
$$v_{p+1}(0)\leqslant 2c^2u_{p+1}(0) +\frac{1}{2} v_{p}(0)\leqslant 2^{-p+1}\bar{c}c^2e^{2c^2}+\frac{1}{2}v_{p}(0).$$
It is easy to check that
\begin{align}\label{v1}
v_{p+1}(0)\leqslant2^{-p}\left(p\bar{c}2c^2e^{2c^2}+v_{1}(0)\right).
\end{align}
For any $q>p$, from Lemma \ref{l2.2}, we know that $X^q(t)\ominus X^p(t)$ and $Z^q(t)\ominus Z^p(t)$ exist and
$$X^q(t)\ominus X^p(t)=X^q(t)\ominus X^{q-1}(t)+X^{q-1}(t)\ominus X^{q-2}(t)+\cdots +X^{p+1}(t)\ominus X^p(t),$$
$$Z^q(t)\ominus Z^q(t)=Z^q(t)\ominus Z^{q-1}(t)+Z^{q-1}(t)\ominus Z^{q-2}(t)+\cdots +Z^{p+1}(t)\ominus Z^p(t).$$
It follows from \eqref{u1}, \eqref{v1} and the triangle inequality that
$$\|X^q(t)\ominus X^p(t)\|_s \leqslant (q-p)2^{-p}\bar{c}e^{2c^2}$$
and
$$\|Z^q(t)\ominus Z^p(t)\|_s \leqslant (q-p)2^{-p}\left(p\bar{c}2c^2e^{2c^2}+v_{1}(0)\right).$$
Thus, Theorem \ref{t32} and Remark \ref{r21} show that there exists a pair  $(X, Z)\in \mathcal{L}^2_{ad}([0,T]\times \Omega,\mathcal{K}(\mathbb{R}^n))\times \mathbb{G}$ such that
$$\|X(t)\ominus X^p(t)\|_s \rightarrow 0,\quad \|Z(t)\ominus Z^p(t)\|_s \rightarrow 0.$$
Now it follows from \eqref{e4+1} that
$$
X(t)+\int_t^TZ(s)dW_s=\xi +\int_t^Tf(s,\omega,Z(s))ds,  \quad \forall t\in [0,T], \quad \mathbb{P}\mbox{-}a.s..
$$

Next we show the uniqueness. Assume that $(X_1(t),Z_1(t))$ and $(X_2(t),Z_2(t))$ are two solutions of \eqref{p1} and $Z_1(t)\ominus Z_2(t)$ exists. Then
$$X_1(t)+\int_t^TZ_1(s)dW_s=\xi +\int_t^Tf(s,Z_1(s))ds,\quad X_2(t)+\int_t^TZ_2(s)dW_s=\xi +\int_t^Tf(s,Z_2(s))ds.$$
It follows from Assumption \ref{c51} (ii) that $f(s,Z_1(s))\ominus f(s,Z_2(s))$ exists. By Lemma \ref{42} and Lemma \ref{5}, we have
\begin{align*}
X_1(t)\ominus X_2(t)&=\mathbb{E}\left[\xi+\int_t^Tf(s,Z_1(s))ds\bigg|\mathcal{F}_t\right]\ominus \mathbb{E}\left[\xi+\int_t^Tf(s,Z_2(s))ds\bigg|\mathcal{F}_t\right]\\
&=\mathbb{E}\left[\int_t^Tf(s,Z_1(s))\ominus f(s,Z_2(s))ds\bigg|\mathcal{F}_t\right],
\end{align*}
which implies that $X_1(t)\ominus X_2(t)$ exists.  This yields that
$$X_1(t)\ominus X_2(t)+\int_t^T Z_1(s)\ominus Z_2(s)dW_s=\int_t^T f(s,Z_1(s))\ominus f(s,Z_2(s))ds,$$
which is equal to
$$X_{1i}(t) \ominus X_{2i}(t)+\int_t^T Z_{1i}(s)\ominus Z_{2i}(s)dW_s= \int_t^Tf_i(s,Z_1(s))\ominus f_i(s,Z_2(s))ds,\quad i=1,\cdots,n, $$
where $X_{1i}(t) \ominus X_{2i}(t),f_i(t,Z_1(t))\ominus f_i(t,Z_2(t)) \in \mathcal{L}^2_{ad}([0,T]\times \Omega,\mathcal{K}(\mathbb{R})$,
$(Z_{1i}(t)\ominus Z_{2i}(t))^T  \in \mathcal{K}_w(\mathcal{L}^2_{ad}([0,T]\times \Omega,\mathcal{K}(\mathbb{R}^m)))$ are the component of $X_1(t)\ominus X_2(t),f(s,Z_1(t))\ominus f(s,Z_2(t))$  and $Z_{1}(t)\ominus Z_{2}(t)$ respectively. It follows from  Corollary \ref{cor3.1} that
\begin{align} \label{e4+6}
&\quad\;\mathbb{E}\|X_{1}(t)\ominus X_{2}(t)\|^2+\mathbb{E}\int_t^T\|Z_{1}(s) \ominus Z_{2}(s)\|^2ds\nonumber
\\
&\leqslant 2\mathbb{E}\int_t^T \|X_{1}(s)\ominus X_{2}(s)\|\|f_i(s,Z_1(s))\ominus f_i(s,Z_2(s))\|ds\nonumber
\\
&\leqslant \frac{1}{2}\mathbb{E}\int_t^T\|Z_{1}(s) \ominus Z_{2}(s)\|^2ds + 2c^2\mathbb{E}\int_t^T\|X_{1}(s)\ominus X_{2}(s)\|^2ds.
\end{align}
Denote
$$u_1(t)=\mathbb{E}\int_t^T\|X_{1}(s)\ominus X_{2}(s)\|^2ds, \quad v_1(t)=\mathbb{E}\int_t^T\|Z_{1}(s)\ominus Z_{2}(s)\|^2ds.$$
Then \eqref{e4+6} leads to
\begin{align}\label{e4.1}
-\frac{d}{dt}\left(u_1(t)e^{2c^2t}\right)+e^{2c^2t}v_1(t)\leqslant \frac{1}{2}e^{2c^2t}v_1(t).
\end{align}
Integrating from $t$ to $T$ for two sides of \eqref{e4.1}, we obtain
$$u_1(t)+\int_t^T e^{2c^2(s-t)}v_1(s)ds \leqslant \frac{1}{2} \int_t^T e^{2c^2(s-t)}v_1(s)ds.$$
This implies that
$$u_1(t)\leqslant \int_t^T e^{2c^2(s-t)}v_1(s)ds,\; v_1(t)\leqslant \frac{1}{2}v_1(t).$$
Thus, $v_1(0)=0$ and $u_1(0)=0$.
\end{proof}

We now consider the following set-valued equation
$$X(t)+\int_t^TZ(s)dW_s=\xi +\int_t^Tf(s,X(s),Z(s))ds,  \quad \quad \mathbb{P}\mbox{-}a.s.\quad t\in [0.T].$$
where $\xi \in \mathbb{L}_T^2(\Omega,\mathcal{K}(\mathbb{R}^n))$, $Z_t \in \mathbb{G}$ and $f: [0,T] \times \Omega  \times \mathcal{K}(\mathbb{R}^n) \times \mathcal{P}(\mathbb{R}^{n\times m}) \rightarrow \mathcal{K}(\mathbb{R}^n)$.
\begin{assumption}\label{c52} Assume $f: [0,T] \times \Omega  \times \mathcal{K}(\mathbb{R}^n) \times \mathcal{P}(\mathbb{R}^{n\times m}) \rightarrow \mathcal{K}(\mathbb{R}^n)$ satisfies the following conditions:
\begin{enumerate}[($\romannumeral1$)]
\item for any given $A\in \mathcal{K}(\mathbb{R}^n)$ and $B\in \mathcal{P}(\mathbb{R}^{n\times m}) $, $f(\cdot,\cdot,A,B) \in \mathcal{L}^2_{ad}([0,T]\times \Omega,\mathcal{K}(\mathbb{R}^n))$;
\item for any fixed $(t,\omega) \in [0,T]\times \Omega$, $A,B \in \mathcal{K}(\mathbb{R}^n)$ and $C,D\in\mathcal{P}(\mathbb{R}^{n\times m})$, the Hukuhara difference $f(t,\omega,A,C) \ominus f(t,\omega,B,C)$ exists whenever $A\ominus B$ exists and the Hukuhara difference $f(t,\omega,B,C)\ominus f(t,\omega,B,D)$ exists whenever $C\ominus D$ exists;
\item there exists a constant $c>0$ such that, for any $A,B \in \mathcal{K}(\mathbb{R}^n)$ and $C,D\in\mathcal{P}(\mathbb{R}^{n\times m})$ with $A\ominus B$ and $C\ominus D$ existing,
$$\|f(t,\omega,A,C)\ominus f(t,\omega,B,D)\| \leqslant c(\|A\ominus B \|+\|C\ominus D\|), \quad \forall t\in [0,T].$$
\end{enumerate}
\end{assumption}

\begin{thm}\label{th4.2}
Let $\xi \in \mathbb{L}_T^2(\Omega,\mathcal{K}(\mathbb{R}^n))$ and $f$ satisfy the Assumption \ref{c52}. Then, there exists a pair $(X,Z)\in \mathcal{L}^2_{ad}([0,T]\times \Omega,\mathcal{K}(\mathbb{R}^n))\times \mathbb{G}$ such that
\begin{align}\label{p2}
X(t)+\int_t^TZ(s)dW_s=\xi +\int_t^Tf(s,X(s),Z(s))ds,  \quad \quad \mathbb{P}\mbox{-}a.s.\quad t\in [0.T].
\end{align}
Moreover, if $(X_1(t),Z_1(t))$ and $(X_2(t),Z_2(t))$ are two solutions to \eqref{p2} with $X_1(t)\ominus X_2(t)$ and $Z_1(t)\ominus Z_2(t)$ existing, then $X_1(t)=X_2(t)$, $Z_1(t)=Z_2(t)$ $\mathbb{P}$-a.s..
\end{thm}
\begin{proof}
We first show the existence of solutions for \eqref{p2}.  For any given $X^0=\left\{0\right\}$ and $Z^0=\left\{0\right\}$, by Lemma \ref{l4.1}, we can obtain a sequence $\{(X^p(t),Z^p(t))\}$ satisfying
\begin{align}\label{e4+7}
X^p(t)+\int_t^TZ^p(s)dW_s=\xi +\int_t^Tf(s,X^{p-1}(s),Z^{p-1}(s))ds, \quad t\in[0,T], p=1,2,\cdots.
\end{align}
We conclude that $X^p(t)\ominus X^{p-1}(t)$, $Z^p(t)\ominus Z^{p-1}(t)$, $f(t,X^p(t),Z^p(t))\ominus f(t,X^{p-1}(t),Z^{p-1}(t))$ exist.
In fact, for $p=1$, it is trivial since $X^0=\left\{0\right\}$ and $Z^0=\left\{0\right\}$.
Assume that the assertion is true for $p-1$ with $p>1$. Then $X^{p-1}(t)\ominus X^{p-2}(t)$ and $Z^{p-1}(t)\ominus Z^{p-2}(t)$ exist. From Assumption \ref{c52} (ii), $f(t,X^{p-1}(t),Z^{p-1}(t))\ominus f(t,X^{p-2}(t),Z^{p-1}(t))$ and $f(t,X^{p-2}(t),Z^{p-1}(t))\ominus f(t,X^{p-2}(t),Z^{p-2}(t))$ exist for any $t\in [0,T]$. It follows from Lemma \ref{l2.2} that $f(t,X^{p-1}(t),Z^{p-1}(t))\ominus f(t,X^{p-2}(t),Z^{p-2}(t))$ exists and
\begin{align*}
&f(t,X^{p-1}(t),Z^{p-1}(t))\ominus f(t,X^{p-2}(t),Z^{p-2}(t))\\
=&f(t,X^{p-1}(t),Z^{p-1}(t))\ominus f(t,X^{p-2}(t),Z^{p-1}(t))+f(t,X^{p-2}(t),Z^{p-1}(t))\ominus f(t,X^{p-2}(t),Z^{p-2}(t)).
\end{align*}
Thus, Lemma \ref{5} yields that
\begin{align}\label{e4+8}
&\mathbb{E}\left[\int_t^T f(s,X^{p-1}(s),Z^{p-1}(s))\ominus f(s,X^{p-2}(s),Z^{p-2}(s))ds\bigg|\mathcal{F}_t\right]\nonumber\\
=&\mathbb{E}\left[\int_t^T f(s,X^{p-1}(s),Z^{p-1}(s)) \bigg|\mathcal{F}_t\right] \ominus \mathbb{E}\left[\int_t^T f(s,X^{p-2}(s),Z^{p-2}(s))ds\bigg|\mathcal{F}_t\right].
\end{align}
By \eqref{e4+7},  one has
$$
\mathbb{E}\left[X^p(t)+\int_t^TZ^p(s)dW_s\bigg|\mathcal{F}_t\right]=\mathbb{E}\left[\xi +\int_t^Tf(s,\omega,X^{p-1}(s),Z^{p-1}(s))ds\bigg|\mathcal{F}_t\right].
$$
It follows from Definition \ref{d25} and Lemma \ref{sitoi} that
\begin{align}\label{e4+9}
X^p(t)&=X^p(t)+\mathbb{E}\left[\int_t^TZ^p(s)dW_s\right]\nonumber\\
&=\mathbb{E}\left[\xi +\int_t^Tf(s,\omega,X^{p-1}(s),Z^{p-1}(s))ds\bigg|\mathcal{F}_t\right],\quad \forall t\in [0,T], \; \mathbb{P}\mbox{-}a.s., \; p=1,2,\cdots.
\end{align}
By \eqref{e4+8} and \eqref{e4+9}, $X^n(t)\ominus X^{n-1}(t)$ exists and
$$X^p(t)\ominus X^{p-1}(t)=\mathbb{E}\left[\xi+\int_t^T f(s,X^{p-1}(s),Z^{p-1}(s))ds\bigg|\mathcal{F}_t\right] \ominus \mathbb{E}\left[\xi+\int_t^T f(s,X^{p-2}(s),Z^{p-2}(s))ds\bigg|\mathcal{F}_t\right].$$
Let
$$M^p(t)=\mathbb{E}\left[\xi+\int_0^T f(s,X^{p-1}(s),Z^{p-1}(s))ds\bigg|\mathcal{F}_t\right].$$
Then by the same argument above, we can shows that $M^p(t)\ominus M^{p-1}(t)$ exists and
$$M^p(t)\ominus M^{p-1}(t)=\mathbb{E}\left[\int_0^T f(s,X^{p-1}(s),Z^{p-1}(s))\ominus f(s,X^{p-2}(s),Z^{p-2}(s))ds\bigg|\mathcal{F}_t\right].$$
By \eqref{e4+7} with $t=0$, we have
$$
\mathbb{E}\left[X^p(0)+\int_0^TZ^p(s)dW_s\bigg|\mathcal{F}_t\right]=\mathbb{E}\left[\xi +\int_0^T f(s,\omega,X^{p-1}(s),Z^{p-1}(s))ds\bigg|\mathcal{F}_t\right].
$$
It follows from Definition \ref{d25} that
\begin{align}\label{e4+10}
M^p(t)=X^p(0)+\int_0^tZ^p(s)dW_s=\mathbb{E}\left[\xi +\int_0^T f(s,\omega,X^{p-1}(s),Z^{p-1}(s))ds\bigg|\mathcal{F}_t\right].
\end{align}
Noting that $M^{p}(t) \ominus M^{p-1}(t)$ is a uniformly square-integrable set-valued martingale, by Lemma \ref{6}, there exists $\tilde{Z}^p(t) \in \mathbb{G}$ such that
$$M^{p}(t) \ominus M^{p-1}(t)=M^{p}(0) \ominus M^{p-1}(0)+\int_0^T \tilde{Z}^p(s) dW_s=X^{p}(0) \ominus X^{p-1}(0)+\int_0^T \tilde{Z}^p(s) dW_s.$$
Moreover, by \eqref{e4+10} and Lemma \ref{l2.11}, one has
\begin{align*}
X^p(0)+\int_0^t Z^p(s)dW_s&=M^{p-1}(t)+M^p(t)\ominus M^{p-1}(t)\\
&=X^{p-1}(0)+X^{p}(0)\ominus X^{p-1}(0)+\int_0^t Z^{p-1}(s)dW_s +\int_0^t \tilde{Z}^p(s)dW_s\\
&=X^{p-1}(0)+X^{p}(0)\ominus X^{p-1}(0)+\int_0^t [Z^{p-1}(s)+\tilde{Z}^p(s)]dW_s, \; \forall t\in [0,T], \; \mathbb{P}\mbox{-}a.s.
\end{align*}
with $p=1,2,\cdots$. From Lemma \ref{equal}, this implies $Z^p(t)=Z^{p-1}(t)+\tilde{Z}^p(t) \in \mathbb{G}$ and so $Z^p(t)\ominus Z^{p-1}(t)=\tilde{Z}^p(t)$ exists.
Thus, it follows from \eqref{e4+9} that
$$X^{p+1}(t)\ominus X^{p}(t)+\int_t^TZ^{p+1}(s)\ominus Z^p(s) dW_s= \int_t^Tf(s,X^{p+1}(s),Z^{p+1}(s))\ominus f(s,X^{p}(s),Z^{p}(s))ds, \quad t\in[0,T],$$
which is equal to
$$X^{p+1}_i(t)\ominus X^{p}_i(t)+\int_t^TZ^{p+1}_i(s)\ominus Z^p_i(s) dW_s= \int_t^Tf_i(s,X^{p+1}(s),Z^{p+1}(s))\ominus f_i(s,X^{p}(s),Z^{p}(s))ds,\quad t\in[0,T],$$
where $$X^{p+1}_i(t)\ominus X^{p}_i(t),\quad f_i(t,X^{p}(t),Z^{n+1}(t))\ominus f_i(t,X^{p-1}(t),Z^{p}(t)) \in \mathcal{L}^2_{ad}([0,T]\times \Omega,\mathcal{K}(\mathbb{R}))$$
and $$(Z^{p+1}_i(t)\ominus Z^p_i(t))^T \in \mathcal{K}_w(\mathcal{L}^2_{ad}([0,T]\times \Omega,\mathcal{K}(\mathbb{R}^m)))$$
are the components of $X^{p+1}(t)\ominus X^{p}(t)$,  $f(t,X^{p}(t),Z^{p+1}(t))\ominus f(t,X^{p-1}(t),Z^{p}(t))$ and $Z^{p+1}(t)\ominus Z^p(t)$,  respectively, $i=1\cdots,n$.

Similar to the proof of Theorem \ref{t52}, from Corollary \ref{cor3.2}, Lemma \ref{sitoi} and Assumption \ref{c52} (ii), we have
\begin{align*}
\mathbb{E}\|X^{p+1}(t)\ominus X^p(t)\|^2&+\mathbb{E}\int_t^T\|Z^{p+1}(s) \ominus Z^p(s)\|^2ds\leqslant 4c^2\mathbb{E}\int_t^T\|X^{p+1}(s) \ominus X^{p}(s)\|^2ds
\\
& \mbox{} + \frac{1}{2}\mathbb{E}\int_t^T\|X^{p}(s)\ominus X^{p-1}(s)\|^2ds+ \frac{1}{2}\mathbb{E}\int_t^T\|Z^{p}(s)\ominus Z^{p-1}(s)\|^2ds, \quad i=1,\cdots,n.
\end{align*}
Denote
$$u_p(t)=\mathbb{E}\int_t^T\|X^p(s)\ominus X^{p-1}(s)\|^2ds,\quad v_p(t)=\mathbb{E}\int_t^T\|Z^{p}(s) \ominus Z^{p-1}(s)\|^2ds, \quad p=1,2,\cdots.$$
Then it follows that
$$-\frac{d(u_{p+1}(t)e^{4c^2t})}{dt}+e^{4c^2t}v_{p+1}(t)\leqslant \frac{1}{2}e^{4c^2t}(u_p(t)+v_p(t)),\quad \forall t\in [0,T], \quad  u_{n+1}(T)=0, \quad p=1,2,\cdots.$$
Integrating from $t$ to $T$ for the two sides of the above inequality, one has
$$u_{p+1}(t)+\int_t^Te^{4c^2(s-t)}v_{p+1}(s)ds \leqslant \frac{1}{2} \int_t^T e^{4c^2(s-t)}u_p(s)ds+\frac{1}{2}\int_t^T e^{4c^2(s-t)}v_p(s)ds,$$
which implies that
\begin{align}
u_{p+1}(t)\leqslant \frac{1}{2} \int_t^T e^{4c^2(s-t)}u_p(s)ds+\frac{1}{2}\int_t^T e^{4c^2(s-t)}v_p(s)ds \label{I1}
\end{align}
and
\begin{align}
\int_t^Te^{4c^2(s-t)}v_{p+1}(s)ds \leqslant \frac{1}{2} \int_t^T e^{4c^2(s-t)}u_p(s)ds+\frac{1}{2}\int_t^T e^{4c^2(s-t)}v_p(s)ds \label{I2}.
\end{align}
Let
$$
c_1=\int_0^T\|Z^1\|ds=\sup\limits_{0\leqslant t\leqslant T}v_1(t),\quad c_2=\int_0^T\|X^1\|ds=\sup\limits_{0\leqslant t\leqslant T}u_1(t).
$$
Iterating the inequalities \eqref{I1} and \eqref{I2}, and taking $t=0$, we have
\begin{align*}
u_{p+1}(0)\leqslant \frac{T(c_1+c_2)}{2^{p-1}}\sum\limits_{k=1}^p\frac{(e^{4c^2T})^k}{k!}, \quad
\int_0^Te^{4c^2s}v_{p+1}(s)ds \leqslant \frac{T(c_1+c_2)}{2^{p-1}}\sum\limits_{k=1}^p\frac{(e^{4c^2T})^k}{k!}.
\end{align*}
Thus, $u_p(0)\to 0$ and $v_p(0)\to 0$. For $q> p$, from Lemma \ref{l2.2}, $X^q(t)\ominus X^p(t)$ and $Z^q(t)\ominus Z^p(t)$ exist and
$$X^q(t)\ominus X^p(t)=X^q(t)\ominus X^{q-1}(t)+X^{q-1}(t)\ominus X^{q-2}(t)+\cdots +X^{p+1}(t)\ominus X^p(t),$$
$$Z^q(t)\ominus Z^p(t)=Z^q(t)\ominus Z^{q-1}(t)+Z^{q-1}(t)\ominus Z^{q-2}(t)+\cdots +Z^{p+1}(t)\ominus Z^p(t).$$
From \eqref{I1}, \eqref{I2} and the triangle inequality, we have
$$\|X^q(t)\ominus X^p(t)\|_s \leqslant (q-p)\frac{T(c_1+c_2)}{2^{p-1}}\sum\limits_{k=1}^p\frac{(e^{4c^2T})^k}{k!}$$
and
$$\|Z^q(t)\ominus Z^p(t)\|_s \leqslant (q-p)\frac{(c_1+c_2)}{2^{p-1}}\sum\limits_{k=1}^p\frac{(e^{4c^2T})^k}{k!}$$
Thus, Theorem \ref{t32} and Remark \ref{r21} show that there exists a pair  $(X, Z)\in \mathcal{L}^2_{ad}([0,T]\times \Omega,\mathcal{K}(\mathbb{R}^n))\times \mathbb{G}$ such that
$$\|X(t)\ominus X^p(t)\|_s \rightarrow 0,\quad \|Z(t)\ominus Z^p(t)\|_s \rightarrow 0.$$
Now \eqref{e4+7} leads to
$$
X(t)+\int_t^TZ(s)dW_s=\xi +\int_t^Tf(s,\omega,X(s),Z(s))ds,  \quad \forall t\in [0,T], \quad \mathbb{P}\mbox{-}a.s..
$$

Next we prove the uniqueness of solutions for \eqref{p2}. Let $(X_1(t),Z_1(t))$ and $(X_2(t),Z_2(t))$ be two solutions. Then
$$X_1(t)+\int_t^TZ_1(s)dW_s=\xi +\int_t^Tf(s,X_1(s),Z_1(s))ds$$
and
$$X_2(t)+\int_t^TZ_2(s)dW_s=\xi +\int_t^Tf(s,X_2(s),Z_2(s))ds.$$
Since $X_1(t)\ominus X_2(t)$, $Z_1(t)\ominus Z_2(t)$ exist, by Assumption \ref{c52} (ii) and Lemma \ref{l2.2}, we know that $f(t,X_1(t),Z_1(t))\ominus f(t,X_2(t),Z_2(t))$ exists    and
\begin{align*}
&f(t,X_1(t),Z_1(t))\ominus f(t,X_2(t),Z_2(t))\\
=&f(t,X_1(t),Z_1(t))\ominus f(t,X_1(t),Z_2(t))+f(t,X_1(t),Z_2(t))\ominus f(t,X_2(t),Z_2(t))
\end{align*}
This implies that
$$X_1(t)\ominus X_2(t)+\int_t^T Z_1(s)\ominus Z_2(s)dW_s=\int_t^T f(s,X_1(s),Z_1(s))\ominus f(s,X_2(s),Z_2(s))ds,$$
which equal to
$$X_{1i}(t) \ominus X_{2i}(t)+\int_t^T Z_{1i}(s)\ominus Z_{2i}(s)dW_s= \int_t^Tf_i(s,X_1(s),Z_1(s))\ominus f_i(s,X_2(s),Z_2(s))ds,\quad i=1,\cdots,n, $$
where
$$X_{1i}(t)\ominus X_{2i}(t),f_i(t,X_1(t),Z_1(t))\ominus f_i(t,X_2(t),Z_2(t)) \in \mathcal{L}^2_{ad}([0,T]\times \Omega,\mathcal{K}(\mathbb{R})$$and$$(Z_{1i}(t)\ominus Z_{2i}(t))^T  \\ \in \mathcal{K}_w(\mathcal{L}^2_{ad}([0,T]\times \Omega,\mathcal{K}(\mathbb{R}^m)))$$
are the component of $X_{1}(t)\ominus X_{2}(t),f(t,X_1(t),Z_1(t))\ominus f(t,X_2(t),Z_2(t))$  and $Z_{1}(t)\ominus Z_{2}(t)$, respectively.  It follows from Corollary \ref{cor3.2}, Lemma \ref{itoi} and Assumption \ref{c52} (ii) that
\begin{align*}
&\quad\;\mathbb{E}\|X_{1}(t)\ominus X_{2}(t)\|^2+\mathbb{E}\int_t^T\|Z_{1}(s) \ominus Z_{2}(s)\|^2ds
\\
&\leqslant \frac{1}{2}\mathbb{E}\int_t^T\|Z_{1}(s) \ominus Z_{2}(s)\|^2ds + (4c^2+\frac{1}{2})\mathbb{E}\int_t^T\|X_{1}(s)\ominus X_{2}(s)\|^2ds, \quad
\end{align*}
Denote
$$u_1(t)=\mathbb{E}\int_t^T\|X_1(s)\ominus X_2(s)\|^2ds,\quad v_1(t)=\mathbb{E}\int_t^T\|Z_1(s) \ominus Z_2(s)\|^2ds.$$
Then it follows that
$$-\frac{d(u_1(t)e^{(4c^2+\frac{1}{2})t})}{dt}+e^{4c^2t}v_1(t)\leqslant \frac{1}{2}e^{(4c^2+\frac{1}{2})t}v_1(t),\quad \forall t\in [0,T].$$
Integrating from $t$ to $T$ for the two sides of the above inequality, we have
\begin{align*}
u_1(t)+\int_t^Te^{(4c^2+\frac{1}{2})(s-t)}v_1(s)ds \leqslant\frac{1}{2}\int_t^T e^{(4c^2+\frac{1}{2})(s-t)}v_1(s)ds
\end{align*}
This implies that $v_1(t)\leqslant \frac{1}{2} v_1(t)$ and $u_1(t) \leqslant  \frac{1}{2}\int_t^T e^{(4c^2+\frac{1}{2})(s-t)}v_1(s)ds$. Thus, we have $v_1(0)=0$ and $u_1(0)=0$.
\end{proof}

\begin{remark}
\begin{enumerate}[($\romannumeral1$)]
\item Theorem \ref{th4.2} reduces to Theorem 5.9 of \cite{Ararat2020Set-valued} when $f(t,X(t),Z(t))\equiv f(t,X(t))$;
\item Theorem \ref{th4.2} gives an answer to an open problem proposed by Ararat et al. \cite{Ararat2020Set-valued}.
\end{enumerate}
\end{remark}


\begin{thebibliography}{99}



\bibitem{Ahmad2006Dynamics}
B. Ahmad, S. Sivasundaram.
\newblock Dynamics and stability of impulsive hybrid setvalued integro-differential equations with delay.
\newblock {\em Nonlinear Anal. TMA}, 2006, 65(11): 2082-2093.

\bibitem{Ahmad2006The}
B. Ahmad, S. Sivasundaram.
\newblock The monotone iterative technique for impulsive hybrid setvalued integro-differential equations.
\newblock {\em Nonlinear Anal. TMA}, 2006, 65(12): 2260-2276.

\bibitem{Alexander2018Mean-Field}
A. Alexander.
\newblock Mean-field type games between two players driven by backward stochastic differential equations.
\newblock {\em Games}, 2018, 9(4): 88.

\bibitem{Al-Hussaini1987Nagoya}
A.N. Al-Hussaini, R.L. Elliott.
\newblock  An extension of It\^{o}'s differentiation formula.
\newblock{\em Nagoya Math. J.}, 1987, 105: 9-18.

\bibitem{Applebaum1984Fermion}
D.B. Applebaum, R.L. Hudson.
\newblock  Fermion It\^{o}'s formula and stochastic evolutions.
\newblock{\em Communications in Mathematical Physics}, 1984, 96(4): 473-496.


\bibitem{Aubin}J.P. Aubin, A. Cellina. {\it Differential Inclusions}. Springer, Berlin, 1984.

\bibitem{Aubin+}J.P. Aubin, H. Frankowska. {\it Set-valued Analysis}. Birkh\"{a}user, Boston, 1990.





\bibitem{Ararat2020Set-valued}
C. Ararat, J. Ma, W.Q. Wu. Set-valued backward stochastic differential equation. 2020, arXiv:2007.15073.

\bibitem{Aswani2019Statistics}
A. Aswani.
\newblock Statistics with set-valued functions applications to inverse approximate optimization.
\newblock {\em Math. Prog.}, 2019, 174(1): 225-251.


\bibitem{Bender2015A}
Bender, Christian, R. Knobloch, P. Oberacker.
\newblock  A generalised It\^o's formula for L\'{e}vy-driven Volterra processes.
\newblock{\em Proces. Appl.}, 2015, 125(8): 2989-3022.

\bibitem{Borkowski2017Forward}
D. Borkowski,
\newblock  Forward and backward filtering based on backward stochastic differential equations.
\newblock {\em Inver. Probl. Imag.}, 2017, 10(2): 305-325.

\bibitem{Catuogno2014Time}
P. Catuogno, C. Olivera.
\newblock Time-dependent tempered generalized functions and It\^{o}'s formula.
\newblock{\em Appl. Anal.}, 2014, 93(3): 539-550.

\bibitem{Deimling}K. Deimling. {\it Multivalued Differential Equations}. Gruyter, Berlin, 1992.

\bibitem{Elliott2012Markovian}
R.J. Elliott, T.K. Siu.
\newblock  Markovian forward-backward stochastic differential equations and stochastic flows.
\newblock {\em Sys. Control Lett.}, 2012, 61(10):1017-1022.

\bibitem{Gradinaru2005m}
M. Gradinaru, I. Nourdin, F. Russo, et al.
\newblock  $m$-order integrals and generalized It\^{o}'s formula: the case of a fractional Brownian motion with any Hurst index.
\newblock{\em  Anna. L'I.H.P. Probab. Statis.}, 2005, 41(4): 781-806.

\bibitem{Hukuhara1967Integration}
M. Hukuhara.
\newblock Integration des applications measurables dont la valeur est un compact convexe.
\newblock {\em Funkcialaj Ekvacioj.}, 1967, 10: 205-223.

\bibitem{Hamel2011Set-valued}
A.H. Hamel, F. Heyde, B. Rudloff.
\newblock Set-valued risk measures for conical market models.
\newblock {\em Math. Finan. Econ.}, 2011, 5(1): 1-28.

\bibitem{Hu2012Backward}
M.S. Hu, S.L. Ji, S.G. Peng, Y.S. Song.
\newblock Backward stochastic differential equations driven by $G$-Brownian motion.
\newblock {\em Stoch. Proces. Appl.}, 2012, 124(1):759-784.

\bibitem{Hamadene1995Zero-sum}
S. Hamadene, J.P. Lepeltier.
\newblock Zero-sum stochastic differential games and backward equations.
\newblock {\em Sys. Control Lett.}, 1995, 24(4): 259-263.

\bibitem{Ito1944Stochastic}
K. It\^{o}.
\newblock Stochastic integral.
\newblock{\em Proc. Imp. Acad. Tokyo}, 1944, 22: 519-524.


\bibitem{Ito1951On}
K. It\^{o}.
\newblock {\it On Stochastic Differential Equations}.
\newblock{Amer. Math. Soc.},  Providence, 1951.

\bibitem{Ito1951Ona}
K. It\^{o}.
\newblock On a formula concerning stochastic differentials.
\newblock{\em Nagoya Math. J.}, 1951, 3:55-65.

\bibitem{Karoui1997Backward}
N.E. Karoui, S.G. Peng , M. C. Quenez.
\newblock Backward stochastic differential equations in finance.
\newblock {\em Math. Finance}, 1997, 7(1):1-71.

\bibitem{Kisielewicz2013Stochastic}
M. Kisielewicz.
\newblock {\it Stochastic Differential Inclusions and Applications}.
\newblock Springer, Berlin, 2013.

\bibitem{Kisielewicz2014Martingale}
M. Kisielewicz.
\newblock Martingale representation theorem for set-valued martingales.
\newblock{\em J. Math. Anal. Appl.}, 2014, 409(1): 111-118.



\bibitem{KR65}K. Kuratowski, C. Ryll-Nardzewski. A general theorem on selectors. {\it Bull. Acad.
Polon. Sci. S\'{e}r. Sci. Math. Astronom. Phys.}, 1965, 13: 397-403.

\bibitem{Lakshmikantham2003Existence}
V. Lakshmikantham, A.A. Tolstonogov.
\newblock  Existence and interrelation between set and fuzzy differential equations.
\newblock {\em Nonlinear Anal. TMA}, 2003, 55(3): 255-268.


\bibitem{Lakshmikantham2006Theory}
V. Lakshmikantham, T.G. Bhaskar, J.V. Devi.
\newblock {\it Theory of Set Differential Equations in a Metric Space}.
\newblock {\em Cambridge Scientific},  Cambridge, 2006.

\bibitem{Li2012Mean-field}
Z. Li, J.W. Luo.
\newblock Mean-field reflected backward stochastic differential equations.
\newblock {\em Stoch. Proces. Appl.}, 2012, 82(11):1961-1968.

\bibitem{Li2010Strong}
J. Li, S. Li, Y. Ogura.
\newblock Strong solution of It\^{o} type set-valued stochastic differential equation.
\newblock {\em Acta Math. Sinica}, 2010, 26(9): 1739-1748.

\bibitem{Peng2009Backward}
S.G. Peng.
\newblock Backward stochastic differential equation driven by fractional Brown motion.
\newblock {\em SIAM J. Control Optim.}, 2009, 48(3): 1675-1700.

\bibitem{Peng1990Adapted}
E. Pardoux, S.G . Peng.
\newblock Adapted solution of a backward stochastic differential equation.
\newblock {\em Sys. Control Lett.}, 1990, 14: 55-61.

\bibitem{Phu2008Stability}
N.D. Phu, L.T. Quang, T.T. Tung.
\newblock  Stability criteria for set control differential equations.
\newblock {\em Nonlinear Anal. TMA}, 2008, 69(11): 3715-3721.

\bibitem{Revuz2008Continuous}
D. Revuz. M. Yor.
\newblock {\it Continuous Martingales and Brownian Motion.}
\newblock Springer, Berlin, 2008.

\bibitem{Tolstonogov}A. Tolstonogov. {\it Differential Inclusions in a Banach Space}.  Springer, Dordrecht, 2000.


\bibitem{Wang2014Linear}
G.C. Wang, H. Xiao.
\newblock Linear quadratic non-zero sum differential games of backward stochastic differential equations with asymmetric information.
\newblock {\em Arxiv:1407.0430}, 2014.

\bibitem{Xu2018Authenticating}
C. Xu, Q. Chen, H.B. Hu, J.L. Xu, X.J. Hei.
\newblock Authenticating aggregate queries over set-valued data with confidentiality.
\newblock {\em IEEE Trans. Knowl.  Data Eng.}, 2018, 30(3): 630-644.



\bibitem{Zaslavski1996Turnpike}
J.A. Zaslavski.
\newblock Turnpike theorem for a class of set-value mappings.
\newblock {\em Numer. Func. Anal. Optim.}, 1996, 17(1-2): 215-240.


\bibitem{Zhang2017Backward}
J.F. Zhang.
\newblock  {\it Backward Stochastic Differential Equations}.
\newblock {\em Springer,  New York}, 2017.




\end{thebibliography}
\end{document}